\documentclass[12pt,reqno,a4paper]{amsart} 
\usepackage{amssymb}
\usepackage[mathscr]{eucal}
\usepackage{xypic}
\usepackage{amsmath}
\usepackage[all, 2cell]{xy}
\usepackage{epsfig}
\usepackage{amscd}
\usepackage{mathrsfs}
\usepackage{xcolor}
\usepackage{hyperref}

\newtheorem{theorem}{Theorem}[section]

\newtheorem{lemma}[theorem]{Lemma}
\newtheorem{proposition}[theorem]{Proposition}

\theoremstyle{definition}
\newtheorem{definition}[theorem]{Definition}

\newtheorem{example}[theorem]{Example}

\newtheorem{remark}[theorem]{Remark}

\numberwithin{equation}{section}

\def\C{{\mathbb C}}

\def\XX{{\mathbb X}}
\def\YY{{\mathbb Y}}

\def\Hom{\mathrm{Hom}}

\def\F{{\mathscr F}}

\def\GL{\mathrm{GL}}

\def\g{{\mathfrak g}}

\def\Y{{\mathcal Y}}

\def\E{{\mathscr E}}

\def\X{{\mathscr X}}
\def\Yy{{\mathscr Y}}

\def\BB{{\mathscr B}}

\def\rra{\rightrightarrows}
\def\mb{\mathbb}
\def\mc{\mathcal}
\def\og{\leavevmode\raise.3ex\hbox{$\scriptscriptstyle\langle\!\langle$~}}
\def\fg{\leavevmode\raise.3ex\hbox{~$\!\scriptscriptstyle\,\rangle\!\rangle$}}

\begin{document}
	
	\title[Atiyah sequences and connections for bundles on Lie groupoids]
	{Atiyah sequences and connections on principal bundles over Lie groupoids and differentiable stacks}
	
	\author[I. Biswas]{Indranil Biswas}
	
	\address{School of Mathematics,
		Tata Institute of Fundamental Research,
		Homi Bhabha Road, Mumbai 400005, India}
	\email{indranil@math.tifr.res.in}
	
	\author[S. Chatterjee]{Saikat Chatterjee}
	
	\address{School of Mathematics,
		Indian Institute of Science Education and Research Thiruvananthapuram,
		Maruthamala P.O., Vihtura, Kerala 695551, India}
	\email{saikat.chat01@gmail.com}
	
	\author[P. Koushik]{Praphulla Koushik}
	
	\address{School of Mathematics,
		Indian Institute of Science Education and Research Thiruvananthapuram,
		Maruthamala P.O., Vihtura, Kerala 695551, India}
	\email{koushik16@iisertvm.ac.in}
	
	\author[F. Neumann]{Frank Neumann}
	
	\address{Pure Mathematics Group, School of Computing and Mathematical Sciences, University of Leicester, University Road, Leicester LE1 7RH, England, UK}
	\email{fn8@le.ac.uk}
	
	\subjclass[2010]{Primary 53C08, Secondary 22A22, 58H05, 53D50}
	
	\keywords{Atiyah exact sequence, principal bundles, Lie groupoids, connections, differentiable stacks}
	
	
\begin{abstract}
We construct and study both general and integrable connections on Lie groupoids and differentiable stacks, as well as on 
principal bundles over them using an Atiyah exact sequence of vector bundles associated to transversal tangential 
distributions.
\end{abstract}
	
	\maketitle
	
	\addtocontents{toc}{\setcounter{tocdepth}{1}} 
	\tableofcontents
	
	\section{Introduction}\label{Intro}

	We develop a general theory of connections for principal bundles over Lie groupoids and 
	differentiable stacks using Atiyah exact sequences associated to transversal tangential distributions. The constructions presented here
	are inspired by the classical work of Atiyah \cite{At} on connections for fiber bundles in complex geometry. 
	
Given a Lie groupoid $\XX \,:=\,[X_1\rightrightarrows X_0]$ with both $X_0$ and $X_1$ smooth manifolds, such that the 
source map $s$ is a submersion, we introduce connections on $\XX$ as a distribution ${\mathcal H}\,\subset\, TX_1$ 
transversal to the fibers of the source map $s$. Such a connection is said to be integrable or flat if, in addition, 
the corresponding distribution $\mathcal H$ is integrable. These connections (respectively, flat connections) also 
give rise to connections (respectively, flat connections) on differentiable stacks under certain compatibility conditions.  For 
the particular case of Deligne--Mumford stacks, which are presented by \'etale Lie groupoids and include descriptions 
for orbifolds and foliations, such a connection always exists. Indeed, for a Deligne--Mumford stack a natural tangential 
distribution is given by the tangent bundle itself.

Now given a Lie group $G$ and a principal $G$-bundle over a Lie 
groupoid $\XX\, :=\,[X_1\rightrightarrows X_0]$ equipped with such a connection ${\mathcal H}\,\subset\, TX_1$, we can 
define the notion of a connection on the principal $G$-bundle. A principal $G$-bundle over the Lie groupoid $\XX$ is 
basically given by a principal $G$-bundle $\alpha\,:\, E_G\,\longrightarrow\, X_0$ and some extra compatibility data 
reflecting the groupoid structure. More precisely, a connection for a principal $G$-bundle over $\XX$ corresponds to 
a splitting of the associated Atiyah exact sequence of vector bundles \cite{At}
$$0\,\longrightarrow\, \text{ad}(E_G)\,\longrightarrow\,\text{At}(E_G)\,\longrightarrow\,
TX_0 \,\longrightarrow\, 0$$
which satisfies the condition of being compatible with the various data of the structures involved;
the details are in Section \ref{Bundles}. This then 
allows us to define the associated curvature and characteristic differential forms for connections on principal 
bundles over Lie groupoids.
	
Furthermore, using adequate groupoid presentations, these constructions extend to a framework that enables us to 
define and study connections and characteristic forms for principal $G$-bundles over differentiable stacks. In the 
particular case of Deligne-Mumford stacks, which are represented by \'etale Lie groupoids, we develop the theory of 
connections intrinsically and construct the Atiyah exact sequence out of the stack data. This relies on the fact that in 
the case of Deligne-Mumford stacks, the associated Atiyah exact sequence is again a sequence of vector bundles in a natural 
way. These constructions also corresponds to related ones in \cite{BMW} for the algebraic geometrical context 
(compare also with \cite{LM}). Though we will work throughout this article mainly in the differentiable setting, we 
remark that most of the concepts and constructions presented here also work equally well in the holomorphic and 
algebraic geometrical setting for the differential geometry of complex analytic and algebraic stacks. An associated 
Chern-Weil theory of characteristic classes for our setting is developed in a related article by the authors 
\cite{BCKN2}. Some of the constructions and results presented here were announced earlier in \cite{BN}.
	
Versions of connections and flat connections on differentiable, complex analytic, and algebraic stacks were also 
introduced using cofoliations on stacks by Behrend \cite{B1}. Independently, Tang in \cite{T} defined in a similar 
fashion flat connections for Lie groupoids, which he called \'etalizations. More recently, Arias Abad and Crainic in 
\cite{AC} introduced general Ehresmann connections for Lie groupoids and relate them to their framework of homotopy 
representations of Lie groupoids. Connections for principal bundles over Lie groupoids using pseudo-connection forms 
were studied earlier also by Laurent-Gengoux, Tu, and Xu \cite{LTX}. They also describe the associated Chern-Weil 
theory. Herrera and Ortiz, \cite{HO}, have informed us that they are currently also developing a similar theory for 
principal 2-bundles over Lie groupoids involving Atiyah LA-groupoids. We refer the reader also to other related work 
on the differential geometry of Lie groupoids and differentiable stacks \cite{FN, TXL, CLW, P, Tr, DE}.
	
	\noindent{\bf Outline and organization of the article.} In the first section (Section \ref{Groupoides}), we introduce the category of smooth spaces over which our stacks and groupoids are constructed and present the basic notions of fibered categories, stacks, and groupoids. We also discuss the relation between differentiable stacks on the one side and Lie groupoids on the other. In the following section (Section \ref{Groupoides2}), we study principal bundles over differentiable stacks and Lie groupoids and their categorical interplay. In Section 
	\ref{Connections} we introduce notions and basic properties of connections on Lie groupoids and differentiable stacks using vertical tangential distributions and compare our constructions with other existing frameworks in the literature. We also 
	construct characteristic differential forms for our general connections. In the final section (Section 
	\ref{Bundles}), we define and study connections on Lie groupoids as splittings of associated Atiyah exact sequences. 
	Finally, we apply the theory to study connections on principal bundles over differentiable stacks and give 
	general characterizations for connections on principal bundles over Deligne-Mumford stacks, which corresponds to 
	\'etale Lie groupoids. 
	
	\section{Smooth spaces, groupoids, and stacks}\label{Groupoides}
	
	In this section, we will recall the main notions and constructions needed for our set-up. For a background on the 
	theory of stacks and its main properties we refer to \cite{B1, BX, Fa, H, N}.
	
	\subsection{Smooth spaces and fibered categories}
	
We shall refer to any of
\begin{itemize}
\item the category of ${\mathcal C}^{\infty}$-manifolds,

\item the category of complex analytic manifolds, and

\item the category of smooth schemes of finite type over the field of complex numbers,
\end{itemize}
as the {\it category $\mathfrak S$ of smooth spaces and smooth maps}. The tangent bundle of any smooth space $X$ will be denoted by $TX$. A {\it smooth space} will mean an object in $\mathfrak S$, and a {\it smooth map} will refer to a morphism in $\mathfrak S$ which is a submersion, and by a submersion, we mean a smooth map whose differential restricted to $T_xX$ is surjective for every point $x$ of the domain $X$. An {\it \'etale map} is a 
	smooth immersion in $\mathfrak S$ (so it is also a submersion).
	
	For a smooth space $X$, the structure sheaf of it will be denoted by ${\mathcal O}_X$. A vector bundle on $X$ will 
	be identified with its sheaf of sections, which is a finitely generated locally free sheaf of ${\mathcal 
		O}_X$-modules. The cotangent bundle of $X$ will be denoted by $\Omega^1_X\,=\,T^*X$, and the $p$-th exterior power 
	of $T^*X$ will be denoted by $\Omega^p_X\,:=\,\bigwedge^p T^*X$. An {\it integrable distribution} on a smooth 
	space $X$ is a subbundle ${\mathcal H}\,\subset\,TX$ of the tangent bundle of $X$ such that its
	annihilator ${\mathcal H}^{\bot}$ 
	generates an ideal in $\bigoplus_{k\geq 0} \Omega^k_X$ which is preserved by the exterior differential of the de Rham 
	complex; equivalently, ${\mathcal H}$ is closed under the Lie bracket operation on vector fields.
	
	The {\it big \'etale site} ${\mathfrak S}_{et}$ on the category $\mathfrak S$ is given by the following 
	Grothendieck topology on $\mathfrak S$. We call a family $\{U_i\longrightarrow X\}$ of morphisms in $\mathfrak S$ 
	with target $X$ a {\it covering family} of $X$, if all smooth maps $U_i\,\longrightarrow \, X$ are \'etale and 
	the total map from the disjoint union $$\coprod_i U_i\longrightarrow X$$ is surjective. This defines a 
	pretopology on $\mathfrak S$ generating a Grothendieck topology, which is known as the {\it big 
		\'etale topology} on $\mathfrak S$ (compare \cite[Expos\'e II]{SGA4} and \cite{Vi}). If either or both of two 
	morphisms $U\,\longrightarrow\, X$ and $V\,\longrightarrow\, X$ in $\mathfrak S$ is a submersion, then their 
	fiber product $U\times_X V\,\longrightarrow\, X$ exists.
	
	\begin{definition}[{Groupoid fibration}]\label{def1}
		A {\it groupoid fibration} over $\mathfrak S$ is a category $\mathscr X$, together with a functor
		$$\pi_{\mathscr X}\,:\,\mathscr X\,\longrightarrow\, \mathfrak S$$ satisfying the following axioms:
		\begin{itemize}
			\item[(i)] For every morphism $V\longrightarrow U$ in $\mathfrak S$, and
			every object $x$ of $\mathscr X$ lying over $U$, there exists an
			arrow $y\longrightarrow x$ in $\mathscr X$ lying over $V\longrightarrow U$.
			
			\item[(ii)] For every commutative triangle $$W\,\longrightarrow\, V\,\longrightarrow\, U$$
			in $\mathfrak S$ and all morphisms $z\,\longrightarrow\, x$
			and $y\,\longrightarrow\, x$ in $\mathscr X$ lying over
			$W\,\longrightarrow \,U$ and $V\,\longrightarrow \,U$ respectively,
			there exists a unique arrow $z\,\longrightarrow\, y$ in $\mathscr X$ lying over
			$W\,\longrightarrow\, V$ such that the composition $$z\,\longrightarrow\,
			y\,\longrightarrow \,x$$ is the morphism $z\,\longrightarrow\, x$.
		\end{itemize}
	\end{definition}
	
	The condition (ii) in Definition \ref{def1} ensures that the object $y$ over $V$, which exists
	by Definition \ref{def1}(i), is unique up to a unique isomorphism. Any choice of such
	an object $y$ is called a {\it pullback} of $x$ via the morphism $f\,:\,
	V\,\longrightarrow\, U$. We will write as usual $y\,=\,x\vert_V$ or $y\,=\,f^*x$.
	
	Let $\mathscr X$ be a groupoid fibration over $\mathfrak
	S$. The subcategory of $\mathscr X$ consisting of all objects lying over a
	fixed object $U$ of $\mathfrak S$ with the morphisms being those lying
	over the identity morphism $id_U$ is called the {\it fiber} or {\it
		category of sections} of $\mathscr X$ over $U$. The fiber
	of $\mathscr X$ over $U$ will be denoted by ${\mathscr X}(U)$. 
	
	Groupoid fibrations over $\mathfrak S$ form a 2-category in which fiber products exist (see \cite{Gr, Vi}). For two groupoid fibrations	
	$$
	\pi_{\mathscr X}\,:\, {\mathscr X}\,\longrightarrow\, \mathfrak S\ ~\ \text{ and }\ ~\
	\pi_{\mathscr Y}\,:\, {\mathscr Y}\,\longrightarrow\, \mathfrak S\, ,
	$$
	the $1$-morphisms from ${\mathscr X}$ to ${\mathscr Y}$ are given by functors $\phi\,:\,
	{\mathscr X}\,\longrightarrow\, {\mathscr Y}$ such that
	$$\pi_{\mathscr Y}\circ \phi \,=\, \pi_{\mathscr X}\, .$$
	The $2$-morphisms are given by natural transformations between these
	projection functor preserving functors.
	
	\begin{example}
		Let $F\,:\, {\mathfrak S}\,\longrightarrow\, (Sets)$ be a presheaf, meaning
		a contravariant functor. We get a groupoid fibration ${\mathscr X}$, where the objects are pairs of the form $(U,\, x)$, 
		with $U$ a smooth space and $x\,\in\, F(U)$, while a morphism $(U, \,x)\,\longrightarrow
		\,(V, \,y)$ is a smooth map $f\,:\, U \,\longrightarrow \,V$ such that $x
		\,=\,y\vert_{F(U)}$, equivalently, $x\,=\,F(f)(y)$. The projection functor is given
		by $$\pi\,:\, {\mathscr X} \,\longrightarrow\, {\mathfrak S}\, ,\, \,\ (U,\, x)
		\,\longmapsto\, U\, .$$
		
		Therefore, any sheaf $F\,:\, {\mathfrak S}\,\longrightarrow\, (Sets)$ gives
		a groupoid fibration over $\mathfrak S$. In
		particular, every smooth space $X$ gives a groupoid fibration $\underline X$ over $\mathfrak S$ as the sheaf represented
		by $X$, in other words,
		$${\underline X}(U)\,=\, Hom_{\mathfrak S}(U,\, X)\, .$$
		To simplify notation, we will identify $\underline X$ with
		the smooth space $X$ without further clarification.
	\end{example}
	
	A category $\mathscr X$ fibered in groupoids over $\mathfrak S$ is
	{\it representable} if there exists a smooth space $X$
	isomorphic to $\mathscr X$ as groupoid fibrations
	over $\mathfrak S$. We call a morphism of groupoid fibrations $\mathscr X \,\longrightarrow\, \mathscr Y$ a {\it representable
		submersion} if for every smooth space $U$ and every morphism
	$U\,\longrightarrow\, {\mathscr Y}$, the fiber product ${\mathscr
		X}\times_{\mathscr Y} U$ is representable, and the induced morphism
	of smooth spaces ${\mathscr X}\times_{\mathscr Y} U \,\longrightarrow
	\,U$ is a submersion.
	
	\subsection{Stacks over smooth spaces}
	Now let us recall the definition of a stack \cite{B1, BX}. We clarify that a stack $\mathscr X$ here will 
	always mean a stack over the big \'etale site ${\mathfrak S}_{et}$ of smooth spaces.
	
\begin{definition}[{Stack}]\label{ds} A groupoid fibration
$\mathscr X$ over $\mathfrak S$ is a {\it stack}
if the following gluing axioms hold with respect to the site ${\mathfrak S}_{et}$:
\begin{itemize}

\item[(i)] Take any smooth space $X$ in $\mathfrak S$, any two objects
$x,\, y$ in ${\mathscr X}$ lying over $X$ and any two isomorphisms $$\phi\, ,\, \psi
\,:\, x\,\longrightarrow\, y$$ over $X$. If the condition $\phi\vert_{U_i}
\,=\,\psi\vert_{U_i}$ holds for all $U_i$ in a covering $\{U_i\longrightarrow X\}$,
then $\phi\,=\,\psi$.

\item[(ii)] Take any smooth space $X$ in $\mathfrak S$, any two objects
$x\, , y \,\in \,{\mathscr X}$ lying over $X$ and any covering
$\{U_i\longrightarrow X\}$ with isomorphisms $$\phi_i\,:\,
x\vert_{U_i}\,\longrightarrow\, y\vert_{U_i}$$ for every $i$. If the condition
$\phi_i\vert_{U_{ij}}\,=\, \phi_j\vert_{U_{i j}}$ holds for
all $i,\,j$, then there exists an isomorphism $\phi\,:\, x\,\longrightarrow\, y$
with $\phi\vert_{U_i}\,=\,\phi_i$ for all $i$.

\item[(iii)] For any smooth space $X$ in $\mathfrak S$, any covering
$\{U_i\longrightarrow X\}$, any family $\{x_i\}$ of objects $x_i$ in the
fiber ${\mathscr X}_{U_i}$ and any family of morphisms
$\{\phi_{ij}\}$, where $$\phi_{ij}\,:\, x_i\vert_{U_{ij}}\, \longrightarrow\,
x_j\vert_{U_{ij}}$$ satisfies the cocycle condition $\phi_{jk}\circ
\phi_{ij}\,=\,\phi_{ik}$ in ${\mathscr X}(U_{ijk})$, there exists an
object $x$ lying over $X$ with isomorphisms $$\phi_i\,:\,
x\vert_{U_i}\,\longrightarrow\, x_i$$ such that $\phi_{ij}\circ \phi_i
\,=\,\phi_j$ in ${\mathscr X}(U_{ij})$.
\end{itemize}
\end{definition}

The isomorphism $\phi$ in Definition \ref{ds}(ii) is unique by Definition \ref{ds}(i). Similarly, from Definition 
\ref{ds}(i) and Definition \ref{ds}(ii), it follows that the object $x$ whose existence is asserted in Definition 
\ref{ds}(iii) is unique up to a unique isomorphism. All pullbacks mentioned in Definition \ref{ds} are only unique up 
to isomorphism, but the properties do not depend on particular choices.

In order to be able to do geometry on stacks, we need to compare them with
smooth spaces and extend the geometry to stacks. 
From now on, we will restrict ourselves to stacks over the category
$\mathfrak S$ of ${\mathcal C}^{\infty}$-manifolds, but we remark that we have incarnations of the analogue
concepts and constructions for the category $\mathfrak S$ of complex analytic manifolds as well as smooth
schemes of finite type over the complex numbers.

\begin{definition}[{Differentiable stack}]
A stack $\mathscr X$ over the site $\mathfrak S_{et}$ is called {\it differentiable}
if there exists a smooth space $X$ in $\mathfrak S$ and a
surjective representable submersion $$x\,:\, X\,\longrightarrow\, {\mathscr X},$$ i.e., 
there exists a smooth space $X$ together with a morphism of
stacks $$x\,:\, X\,\longrightarrow\, {\mathscr X},$$ such that for every smooth
space $U$ and every morphism of stacks $U\,\longrightarrow\, {\mathscr X}$, the following two hold:
\begin{enumerate}
\item the fiber product $X\times_{\mathscr X} U$ is representable, and
			
\item the induced morphism of smooth spaces $X\times_{\mathscr X} U\,
\longrightarrow\, U$ is a surjective submersion.
\end{enumerate}
\end{definition}
	
	If $\mathscr X$ is a differentiable stack, such a surjective representable submersion $x\,:\,
	X\,\longrightarrow\, {\mathscr X}$ is called a {\it presentation} or {\it atlas} for the stack
	$\mathscr X$. It need not be unique, in other words, a differentiable stack can have different
	presentations.
	
	\begin{definition}
		A differentiable stack $\mathscr X$ is a {\it (proper) Deligne--Mumford stack} if it has a (proper)
		\'etale presentation.
	\end{definition}
	
	Orbifolds correspond to proper Deligne--Mumford stacks (see \cite{B1, LM}). 
	
	The incarnations of these concepts over the site $\mathfrak S_{et}$ of complex analytic manifolds
	respectively, smooth schemes
	of finite type over $\C$ will be referred to as complex analytic stacks, respectively, algebraic stacks.
	
	\subsection{Lie groupoids and differentiable stacks} Differentiable stacks are also incarnations of Lie groupoids, 
	as we will recall now in detail (see also \cite{BX, MDH}).
	
	\begin{definition}[Lie groupoid]
		A {\it Lie groupoid} $\mathbb X\,=\,[X_1\rightrightarrows X_0]$ is a groupoid
		internal to the category $\mathfrak S$ of smooth spaces, meaning the space $X_1$ of
		arrows and the space $X_0$ of objects are objects of $\mathfrak S$
		and all structure morphisms
		$$s,\, t\,:\, X_1\,\longrightarrow\, X_0\, , \ \ m\,:\,
		X_1\times_{s,X_0,t} X_1\longrightarrow X_1\, ,
		$$
		$$
		i\,:\, X_1\,\longrightarrow\, X_1\, ,\ \ e\,:\, X_0\,\longrightarrow \,X_1
		$$
		are morphisms in $\mathfrak S$ (so they are smooth maps). Here $s$ is the source map, $t$ is the target map, $m$ is the multiplication map, $e$ is the identity section, and $i$ is the inversion map of the groupoid. The source map $s$ is a submersion. Using $i$, this implies that the target map $t$ is also a submersion.
		
		If $s$ and $t$ are \'etale, the groupoid $\XX\,=\,[X_1\rightrightarrows X_0]$ is
		called {\it \'etale}. If the {\it anchor map}
		$$(s,\, t)\,:\, X_1\,\longrightarrow\, X_0\times X_0$$ is {\it proper}, the groupoid
		is called a {\it proper} groupoid. If the
		
		A {\it Lie group} $G$ is a Lie groupoid $[G\,\rightrightarrows\, *]$ with one
		object, meaning the space $X_0$ is just a point in $\mathfrak S$.
	\end{definition}
	
	Every Lie groupoid $\XX\,=\,[X_1\rightrightarrows X_0]$ gives rise to the associated tangent
	groupoid $T\XX\,:=\,[TX_1\rightrightarrows TX_0]$. 
	
	\begin{example}
		Let $X$ be a smooth space. The groupoid fibration
		$\underline X$ is, in fact, a differentiable stack over $\mathfrak S$.
		A presentation is given by the identity morphism $id_X$.
	\end{example}
	
	\begin{example}[Classifying stack]\label{ex1}
		For a Lie group $G$, let ${\mathscr B}G$ be the category which
		has as objects all pairs $(P,\, S)$, where $S$ is a smooth space of
		$\mathfrak S$ and $P$ is a principal $G$-bundle over $S$; a
		morphism $(P,\, S)\,\longrightarrow\, (Q,\, T)$ is a commutative diagram
		\[
		\xy \xymatrix{ P\ar[r]^{\varphi}\ar[d]&
			Q\ar[d]\\
			S\ar[r]& T}
		\endxy
		\]
		where $\varphi\,:\, P\,\longrightarrow\, Q$ is a $G$-equivariant map.
		Note that the above diagram is Cartesian.
		Then ${\mathscr B}G$ together with the projection functor $$\pi:
		{\mathscr B}G\,\longrightarrow \,{\mathfrak S}\, ,\ ~ (P,\, S)\,\longmapsto\, S$$
		is a groupoid fibration over $\mathfrak S$. We note that ${\mathscr B}G$ is
		in fact a differentiable stack, and it is known as the {\it classifying stack} of $G$ (see also Example \ref{BX stack} below).
		A presentation is given by the representable surjective submersion
		$*\,\longrightarrow\, {\mathscr B}G$ where $*$ is a ``point" in $\frak{S}$.	
	\end{example}
	
	\begin{definition}\label{defXs}
		Let ${\mathbb X}\,=\,[X_1\rightrightarrows X_0]$ be a Lie groupoid. A \textit{(left) 
			$\mathbb X$-space} is given by an object $P$ of $\mathfrak S$ together with a
		smooth map $\pi\,:\, P\,\longrightarrow\, X_0$ and a map $$\sigma\,:\, Q\,:=\,
		X_1\times_{s, X_0, \pi} P \,
		\longrightarrow\, P\, , \ \ ~\sigma (\gamma,\, x)\,:=\,\gamma\cdot x$$
		such that,
\begin{itemize}
\item[(i)] $\pi(\gamma\cdot x)\,=\,t(\gamma)$ for all $(\gamma,\,x)\,\in\, X_1\times_{s,X_0,\pi}P$,

\item[(ii)] $e(\pi(x))\cdot x\,=\,x$ for all $x\,\in\, X_1$, and

\item[(iii)] $(\delta \cdot \gamma)\cdot x\,=\, \delta\cdot (\gamma \cdot x)$ for all 
$(\gamma,\, \delta,\,x)\, \in\, X_1\times_{s,X_0,t}X_1\times_{s,X_0,\pi}P$.
\end{itemize}
Similarly, we can also define a \textit{(right) $\mathbb X$-space}.
\end{definition}
	
	Let $\mathscr X$ be a differentiable stack with a given
	presentation $x\,:\, X\,\longrightarrow\, {\mathscr X}$. We can associate to
	$\mathscr X$ a Lie groupoid ${\mathbb X}\,=\,[X_1\rightrightarrows
	X_0]$ as follows: Let $X_0\,:=\,X$ and $X_1\,:=\,X\times_{\mathscr X} X$ and the source and
	target morphisms $$s,\,\,t\,:\, X\times_{\mathscr X} X\,\rightrightarrows\, X$$
	of $\mathbb X$ being the first and second canonical projection morphisms.
	The composition of morphisms $m$ in $\mathbb X$ is given as
	projection to the first and third factor
	$$X\times_{\mathscr X} X\times_{\mathscr X} X \,\cong\, (X\times_{\mathscr X} X)\times_X
	(X\times_{\mathscr X} X)\,\longrightarrow \,X\times_{\mathscr X} X\, .$$ The
	morphism $X\times_{\mathscr X} X\,\longrightarrow\,
	X\times_{\mathscr X} X$ that interchanges the two factors gives the inverse morphism $i$
	for the groupoid, while the unit morphism
	$e$ is given by the diagonal morphism $X\,\longrightarrow\, X\times_{\mathscr X} X$.
	As a presentation $x\,:\, X\,\longrightarrow \,{\mathscr X}$ of a
	differentiable stack is a submersion, it follows that the source and
	target morphisms $$s,\,\, t\,:\, X\times_{\mathscr X} X\,\rightrightarrows\, X,$$
	being induced maps from the fiber product, are also submersions.
	
	In the opposite direction, given a Lie groupoid $\mathbb X$ we can associate a differentiable stack $\mathscr X$
	to it. Basically this is a generalization of associating to a Lie group $G$ its classifying stack ${\mathscr
		B}G$ (see Example \ref{ex1}). For this we now define (compare also \cite{BX}):
	
\begin{definition}[{Torsors}]
		Let ${\mathbb X}\,=\, [X_1\rightrightarrows X_0]$ be a Lie groupoid, and let
		$S$ be a smooth space. A {\it (right)} $\mathbb X$-{\it torsor over} $S$ is a smooth
		space $P$ together with a surjective submersion $\pi\,:\,
		P\,\longrightarrow\, S$ and a right action of $\mathbb X$ on $P$ (see
		the second part of Definition \ref{defXs}) satisfying the condition that $\pi$ is
		$\mathbb{X}$-invariant, i.e., $\pi(\gamma\cdot p)\,=\,\pi(p)$ for all $(\gamma,\,x)
\,\in\, X_1\times_{s,X_0,a}P$ and the extra condition that for
all $p,\,\, p'\,\in\, P$ with $\pi(p)\,=\,\pi(p')$, there exists a unique $x\,\in\,
X_1$ such that $p\cdot x$ is defined and $p\cdot x\,=\,p'$.
		
		Let $\pi\,:\,
		P\,\longrightarrow\, S$ and $\rho\,:\, Q\,\longrightarrow\, T$ be
		$\mathbb X$-torsors. A {\it morphism} of $\mathbb X$-torsors from the
		first one to the second one is given by a Cartesian diagram of smooth maps
		\begin{equation}\label{f1}
		\xy \xymatrix{ P\ar[r]^\varphi\ar[d]&
			Q\ar[d]\\
			S\ar[r]& T}
		\endxy
		\end{equation}
		such that $\varphi$ is a $X_1$-equivariant map. Again, there is also a similar notion of a {\it (left)}
		$\mathbb X$-{\it torsor} over a smooth space $S$.
		
		Given a $\mathbb X$-torsor $P$, the map $a:P\,\longrightarrow\, X_0$ is called
		the {\it anchor map} or {\it momentum map}. The surjective submersion $\pi\,:\, P\,
		\longrightarrow \,S$ is called the {\it structure map}.
	\end{definition}
	
	For a Lie groupoid ${\mathbb X}\,=\,[X_1\rightrightarrows X_0]$, let
	${\mathscr B}{\mathbb X}$ denote the category of $\mathbb X$-torsors,
	meaning the category whose objects are pairs $(P,\, S)$, where $S$ is a smooth space
	and $P$ a $\mathbb X$-torsor over $S$. The morphisms $(P,\, S)\,\longrightarrow\,
	(Q,\, T)$ are given by Cartesian diagrams as in \eqref{f1}. Then
	${\mathscr B}{\mathbb X}$ is a groupoid fibration over $\mathfrak
	S$ with a canonical projection functor
	$$\pi\,:\, {\mathscr B}{\mathbb X}\,\longrightarrow\, {\mathfrak S}\, , \
	~ (P,\, S)\,\longmapsto\, S\, .$$
	It turns out that ${\mathscr B}{\mathbb X}$ is, in fact, a differentiable stack, the
	{\it classifying stack of $\mathbb X$-torsors}. The proof of
	\cite[Proposition 2.3]{BX} works in any of the three categories $\mathfrak S$.
	
	\begin{example}\label{BX stack}
		In the case where the Lie groupoid ${\mathbb X}\,=\,[G\rightrightarrows *]$ is a
		Lie group, a $\mathbb X$-torsor over a smooth space $S$ is simply a $G$-torsor or
		principal $G$-bundle over $S$. The associated classifying stack is the classifying stack
		${\mathscr B}G$ of the Lie group $G$.
	\end{example}
	
	As presentations for a given differentiable stack are not unique, the associated Lie groupoids might be different. In order to define
	algebraic and geometric invariants for differentiable stacks, like cohomology or differential forms,
	they should, however, be independent from the choice of
	a presentation for the given stack. Therefore, it is important to know under which conditions
	two different Lie groupoids give rise to isomorphic differentiable stacks.
	
	\begin{definition}\label{dme}
		Let ${\mathbb X}\,=\,[X_1\rightrightarrows X_0]$ and ${\mathbb
			Y}\,=\,[Y_1\rightrightarrows Y_0]$ be Lie groupoids. A {\it morphism} of Lie groupoids is a functor $\phi\,:\, {\mathbb X}\,\longrightarrow\,{\mathbb Y}$ given by two smooth maps $\phi\,=\,(\phi_1,\, \phi_0)$ with
		$$\phi_0\,:\, X_0\,\longrightarrow\, Y_0\, , \ \ \phi_1\,:\, X_1\,\longrightarrow\, Y_1$$
		which commute with all structure morphisms of the groupoids. 
		
		A morphism $\phi\,:\, {\mathbb X}\,\longrightarrow\, {\mathbb Y}$ of Lie groupoids 
		is an {\it \'etale morphism}, if both maps $\phi_0\,:\, X_0\,\longrightarrow\, Y_0$ and $\phi_1\,:\, X_1\,\longrightarrow\, Y_1$ are \'etale.
		
		A morphism $\phi\,:\, {\mathbb X}\,\longrightarrow\, {\mathbb Y}$ of Lie groupoids is a {\it Morita morphism} or {\it essential equivalence} if
		\begin{itemize}
			\item[(i)] $\phi_0\,:\, X_0\,\longrightarrow\, Y_0$ is a surjective submersion, and
			\item[(ii)] the diagram
			\[
			\xy \xymatrix{ X_1\ar[r]^{(s,t)\,\,\,} \ar[d]_{\phi_1}&
				X_0\times X_0\ar[d]^{\phi_0 \times \phi_0}\\
				Y_1\ar[r]^{(s,t)\,\,\,}& Y_0\times Y_0}
			\endxy
			\]
			is Cartesian, or in other words, $X_1\,\cong\, Y_1\times_{Y_0\times Y_0} (X_0\times X_0)$.
		\end{itemize}
		Two Lie groupoids $\mathbb X$ and $\mathbb Y$ are {\it Morita
			equivalent}, if there exists a third Lie groupoid $\mathbb W$ and
		Morita morphisms
		$${\mathbb X}\,\stackrel{\phi}{\longleftarrow}\, {\mathbb W}
		\,\stackrel{\psi}{\longrightarrow}\, {\mathbb
			Y}\, .$$
	\end{definition}
	
	\begin{theorem}[Morita equivalence]\label{thm-m}
		Let ${\mathbb X}\,=\,[X_1\rightrightarrows X_0]$ and ${\mathbb Y}\,=\,[Y_1\rightrightarrows Y_0]$ be Lie groupoids. Let $\mathscr X$
		and $\mathscr Y$ be the associated differentiable stacks, i.e.,
		$\mathscr X$ is the classifying stack ${\mathscr B}{\mathbb X}$ of\, $\mathbb X$-torsors and
		$\mathscr Y$ is the classifying stack ${\mathscr B}{\mathbb Y}$ of\,
		$\mathbb Y$-torsors. Then the following three statements are equivalent:
		\begin{itemize}
			\item[(i)] The differentiable stacks $\mathscr X$ and $\mathscr Y$ are
			isomorphic.
			
			\item[(ii)] The Lie groupoids $\mathbb X$ and $\mathbb Y$ are Morita equivalent.
			
			\item[(iii)] there exists a smooth space $Q$ together with two smooth maps $f\,:\, Q\,\longrightarrow\, X_0$ and 
			$g\,:\, Q\,\longrightarrow\, Y_0$ and (commuting actions) of $X_1$ and $Y_1$ in such a way that $Q$ is at the same 
			time a left $\XX$-torsor over $Y_0$ via $g$ and a right $\mathbb Y$-torsor over $X_0$ via $f$, in other
			words, $Q$ is a $\XX$-$\mathbb Y$-bitorsor.
		\end{itemize}
	\end{theorem}
	
	\begin{proof}
		For the differentiable category of ${\mathcal C}^{\infty}$-manifolds, this is \cite[Theorem 2.2]{BX}. It is immediate to see that the proof works verbatim for any of the other categories $\mathfrak S$ of smooth spaces, 
		meaning for complex analytic as well as algebraic stacks.
	\end{proof}
	
	Therefore different presentations of the same differentiable stack are given by Morita equivalent Lie groupoids.
	Conversely, Morita equivalent Lie groupoids present isomorphic differentiable stacks. The case of \'etale groupoids presenting Deligne-Mumford stacks is  
	of particular importance as these can be used naturally also to describe orbifolds and foliations \cite{C, CM, L}.
	
	\section{Principal bundles over differentiable stacks and Lie groupoids}\label{Groupoides2}
	
	In this section, we will define the notion of a principal $G$-bundle over a differentiable stack and over a
	Lie groupoid. 
	
	\subsection{Principal bundles over differentiable stacks}
	
	Let us start with the definition of principal bundles over differentiable stacks.
	
	\begin{definition}[Principal bundle over differentiable stack]
		Let $G$ be a Lie group and $\mathscr X$ a differentiable stack. A {\it principal $G$-bundle} or
		{\it $G$-torsor} $\E_G$ over $\mathscr X$ is given by a differentiable stack $\mathscr{E}_G$, a morphism of
		differentiable stacks $\pi\,:\,\mathscr{E}_G\,\longrightarrow\, \mathscr{X}$, and a $2$-Cartesian diagram
		$$
		\xymatrix{\E_G \times G\rto^{\sigma}\dto^{p_1}&\E_G\dto^{\pi}\ar@{=>}[dl] \\
			\E_G\rto^{\pi}&\mathscr X.}
		$$
		such that for any submersion $$f\, :\, U\, \longrightarrow\, \mathscr X$$ from a smooth space $U$, the
		pullback by $f$ in the above diagram defines a principal $G$-bundle on $U$. {\em Morphisms} $\rho\,:\, \E_G
		\,\longrightarrow \, \F_G$ between two principal $G$-bundles $\E_G$ and $\F_G$ are defined in the obvious way.
	\end{definition}
	
	Principal $G$-bundles over a differentiable stack $\mathscr X$ can also be defined directly by using a presentation or atlas. 
	In other words, a principal $G$-bundle $\E_G$ over a differentiable stack $\mathscr X$ is given
	by a principal $G$-bundle $E_G\,\longrightarrow\, X_0$ for an atlas $X_0\,
	\longrightarrow\, {\mathscr X}$ together with an isomorphism of the pullbacks
	$p_1^*E_G\,\stackrel{\sim}\longrightarrow\, p_2^*E_G$ on the fiber product $X_0\times_{\mathscr X} X_0$
	satisfying the cocycle condition on $X_0\times_{\mathscr X} X_0\times_{\mathscr X} X_0$.
	It turns out that for any submersion $f\,:\, U\,\longrightarrow\, {\mathscr X}$, this
	datum defines a principal $G$-bundle $E_G\,\longrightarrow\, U$ over $U$, because
	$X_0\times_{\mathscr X}U\,\longrightarrow\, U$ has local sections. Therefore we get a
	differentiable stack $\E_G$ and the $G$-multiplication map glues and comes with a natural
	morphism of stacks ${\E_G}\times G\,\longrightarrow\, {\E_G}$ (see \cite{BMW, H, N}). We can reformulate all this more explicit also as follows: For each smooth atlas $u\,:\, U\,\longrightarrow\, {\mathscr X}$ we are given a principal
	$G$-bundle $\E_{G, u}$ over $U$ and for each $2$-commutative diagram of the form 
	\begin{equation}\label{diagat}
	\xy \xymatrix{U\ar[dr]_{u}\ar[rr]^{\varphi}& & V\ar[dl]^v\\
		&{\mathscr X }&}
	\endxy
	\end{equation}
	with a $2$-isomorphism $\alpha\,:\, u\,\Rightarrow \,v \circ \varphi$, where $u,\, v$ are smooth atlases, we are given an isomorphism
	$$\theta_{\varphi, \alpha}\,\,\,:\,\,\E_{G, u}\,\stackrel{\cong}{\longrightarrow}\, \varphi^* \E_{G, v}$$
	satisfying the cocycle condition which says that for each $2$-commutative diagram of the form
	\[
	\xy \xymatrix{U\ar[dr]_u\ar[r]^{\varphi}&V\ar[r]^{\psi} \ar[d]^v & W\ar[dl]^w\\
		&{\mathscr X}&}
	\endxy
	\]
	together with $2$-isomorphisms $\alpha\,:\, u\,\Rightarrow\, v \circ \varphi$ and $\beta\,:\, v\,
	\Rightarrow\, w \circ \psi$ we have a commutative diagram
	\[
	\xy \xymatrix{\E_{G,u} \rrto^{\theta_{\psi\circ\varphi, \beta\circ\alpha}}_{\cong}\dto_{\theta_{\varphi, \alpha}}^{\cong}&&(\psi\circ\varphi)^*\E_{G, w}\ar@{=}[d] \\
		\varphi^*\E_{G, v}\rrto^{\cong}_{\varphi^*\theta_{\psi, \beta}}& & (\varphi^*\circ \psi^*)\E_{G, w}}
	\endxy
	\]
	If $\F_G$ is another principal $G$-bundle over $\mathscr X$, then a morphism $f\,:\, \E_G\,\longrightarrow\,
	\F_G$ is given by a morphism $f_u\,:\, \E_{G, u}\, \longrightarrow\, \F_{G, u}$ for each smooth atlas
	$u\,:\, U\,\longrightarrow \,{\mathscr X}$ such that for any $2$-commutative diagram of
	the form (\ref{diagat}) we have a commutative diagram
	\[
	\xy \xymatrix{\E_{G,u} \rrto^{f_u}\dto_{\theta^{\E_G}_{\varphi, \alpha}}^{\cong}&&\F_{G, u}\dto^{\theta^{\F_G}_{\varphi, \alpha}}_{\cong} \\
		\varphi^*\E_{G, v}\rrto_{\varphi^*f_v}& & \varphi^*\F_{G, v}}
	\endxy
	\]
	The category of principal $G$-bundles over a differentiable stack $\mathscr X$ forms in a natural way a groupoid fibration ${\mathscr B}un_G({\mathscr X})$ over $\mathfrak S$. In fact, we get the following characterization from the above (compare also \cite{BMW, H}) 
	
	\begin{proposition}\label{propp}
		Let $\mathscr X$ be a differentiable stack and $G$ a Lie group with classifying stack ${\mathscr B}G$. Giving a principal 
		$G$-bundle over $\mathscr X$ is equivalent to giving a morphism of stacks ${\mathscr X}\,\longrightarrow\,
		{\mathscr B}G$ and two principal $G$-bundles over $\mathscr X$ are isomorphic if and only
		if the corresponding morphisms of stacks ${\mathscr X}\,\longrightarrow\, {\mathscr B}G$ are $2$-isomorphic.
	\end{proposition}
	
	We consider the groupoid fibration ${\mathscr B}un_G({\mathscr X})$ over $\mathfrak S$ whose objects over a 
	smooth space $U$ are principal $G$-bundles ${\mathscr E}_G$ over ${\mathscr X}\times U$ and whose morphisms are given by pullback diagrams of principal $G$-bundles. Given two differentiable stacks $\mathscr X$ and $\mathscr Y$, we also have the groupoid fibration ${\mathscr H}om(\X,\, \Yy)$ over $\mathfrak S$, 
	whose groupoid of sections over a smooth space $U$ is the groupoid of $1$-morphisms or functors ${\mathscr H}om(\X,\,
	\Yy)(U)\,=\,\Hom_U(\X\times U,\, \Yy\times U)$. Let us remark that if in addition $\X$ is proper and $\Y$ of finite 
	presentation, then ${\mathscr H}om(\X,\, \Yy)$ is again a differentiable stack and in particular if $\Yy$ is a Deligne--Mumford 
	stack, then the Hom stack ${\mathscr H}om(\X, \,\Yy)$ is also a Deligne--Mumford stack (compare \cite{Ao}, \cite{Ol}). We can 
	now make the following straightforward observation using Proposition \ref{propp} (compare with \cite{BMW}).
	
	\begin{proposition}
		There is an equivalence of groupoid fibrations over $\mathfrak S$
		$${\mathscr B}un_G({\mathscr X})\,\cong\, {\mathscr H}om(\X, \,{\mathscr B}G)\, .$$
	\end{proposition}
	
	\subsection{Principal bundles over Lie groupoids}
	Let us now recall the general notion of a principal $G$-bundle over a Lie groupoid (see \cite{LTX, TXL}).
	
	Given a $\mathbb X$-space $\pi\,:\, P\,\longrightarrow\, X_0$ for a Lie groupoid
	${\mathbb X}\,=\,[X_1\rightrightarrows X_0]$, we have for any $\gamma\,\in\,X_1$ a smooth isomorphism in $\mathfrak S$
	$$l_{\gamma}\,:\, \pi^{-1}(u)\,\longrightarrow\, \pi^{-1}(v)\, ,\ ~ x\,
	\longmapsto\, \gamma\cdot x\, ,$$
	where $u\,=\,s(\gamma)$ and $v\,=\,t(\gamma)$. Associated to this $\mathbb X$-space
	is a transformation groupoid ${\mathbb P}\,=\,[Q\,=\,X_1\times_{s, X_0, \pi} P\rightrightarrows P]$, where
	the source and target maps are given by $s(\gamma,\, x)\,=\,x$ and
	$t(\gamma, x)\,=\,\gamma\cdot x$. The multiplication map is given by
	$$(\gamma,\, y)\cdot (\delta,\, x)\,=\,(\gamma\cdot \delta,\, x)\, ,$$
	where $y\,=\,\delta\cdot x$. The first projection defines a strict homomorphism of Lie groupoids from
	${\mathbb P}\,=\,[Q\rightrightarrows P]$ to ${\mathbb X}\,=\,[X_1\rightrightarrows X_0]$.
	
	\begin{definition}[Principal bundle over Lie groupoid]\label{PrinGGr}
		Let $G$ be a Lie group and ${\mathbb X}\,=\,[X_1\rightrightarrows X_0]$ a Lie groupoid. A {\em
			principal $G$-bundle} or {\em $G$-torsor} over $\mathbb X$, denoted by
		${\mathbb E}_G\,:=\,[s^*E_G\rightrightarrows E_G]$,
		is given by a principal (right) $G$-bundle $\pi\,:\, E_G\,\longrightarrow\, X_0$, which
		is also a $\mathbb X$-space such that for all $x\,\in\, E_G$ and all $\gamma\,\in\,
		X_1$ with $s(\gamma)\,=\,\pi(x)$, we have $$(\gamma\cdot x)\cdot g\,=\,
		\gamma\cdot (x\cdot g)~ \ \mathrm{for\,\,all} ~\ g\,\in\, G\, .$$
		Let $s^*E_G\,=\,X_1\times_{s, X_0, \pi} E_G$ be the pullback along the source map $s$. Then
		${\mathbb E}_G\,=\,[s^*E_G\rightrightarrows E_G]$ is in a natural way a Lie groupoid, the 
		transformation groupoid with respect to the $X_1$-action.
	\end{definition}
	
	\begin{example}\label{Ex:G-Hequi}
		Let $G$ and $H$ be Lie groups, and let ${\mathbb X}\,=\,[H\times X_0\rightrightarrows
		X_0]$ be the transformation groupoid over a smooth space $X_0$. Then a principal
		$G$-bundle over $\mathbb X$ is an $H$-equivariant principal $G$-bundle over $X_0$.
		In particular, if ${\mathbb X}\,=\,[X_0\rightrightarrows X_0]$ with both structure maps $s, t$
		being the identity map $id_{X_0}$, then a principal $G$-bundle over $\mathbb X$ is just a
		principal $G$-bundle over $X_0$. 
	\end{example}
	
	\begin{example}
		Let $G,\, H$ be a pair of Lie groups, and ${\mathbb X}\,=\,[H\rightrightarrows *]$ the single object  Lie groupoid associated to the Lie group $H$. Then a left-action of $H$ on $G$ satisfying 
		$h\cdot(gg')\,=\, (h\cdot g)g'$ for all $h\, \in\, H$ and $g,\, g'\,\in\, G$ defines a principal $G$-bundle over ${\mathbb X}$, and vice versa. 
	\end{example}

	Similarly, we can define vector bundles of rank $n$ over a Lie groupoid $\XX$ and over a
	differentiable stack $\mathscr X$. They can be identified with the principal
	$\GL_n$-bundles over $\XX$ and $\mathscr X$ respectively.
	
	The following theorem generalizes \cite[Proposition 4.1]{BX} for arbitrary principal bundles and 
	the proof given below is a variation of the argument for $S^1$-bundles. 
	
	\begin{theorem}\label{thm1}
		Let $G$ be a Lie group, $\mathscr X$ be a differentiable stack and ${\mathbb X}=[X_1\rightrightarrows X_0]$ be
		a Lie groupoid presenting $\mathscr X$. Then there is a canonical equivalence of categories:
		$${\mathscr B}un_G({\mathscr X})\,\cong\, {\mathbb B}un_G({\mathbb X})\, ,$$
		where ${\mathscr B}un_G({\mathscr X})$ is the category of principal $G$-bundles over $\mathscr X$
		and ${\mathbb B}un_G({\mathbb X})$ is the category of principal $G$-bundles over $\mathbb X$.
	\end{theorem}
	
	\begin{proof}
		Let us first assume we have given a principal $G$-bundle $\pi\,:\,\E_G\,\longrightarrow\,\X$ over the
		differentiable stack $\X$. Let $E_G$ be the pullback of $X_0\, \longrightarrow\, \X$ via $\pi$,
		in other words, we
		have a $2$-Cartesian diagram
		\[
		\xy \xymatrix{E_G\ar[r] \ar[d]& X_0\ar[d]\\
			\E_G\ar[r]^\pi & \X}
		\endxy
		\]
		This means that $E_G\,\longrightarrow\, X_0$ is a principal $G$-bundle over $X_0$ and $E_G\,
		\longrightarrow\,\E_G$ is a representable surjective submersion.
		Let ${\mathbb E}_G\,=\,[s^*E_G\rightrightarrows E_G]$ be the associated Lie groupoid, which
		comes with an induced morphism of Lie groupoids
		$${\mathbb E}_G\,=\,[s^*E_G\rightrightarrows E_G]\,\longrightarrow\, {\mathbb X}\,=\,
		[X_1 \rightrightarrows X_0]$$ and gives rise to the Cartesian diagram
		\[
		\xy \xymatrix{s^*E_G\ar[r] \ar[d]& X_1\ar[d]\\
			E_G\ar[r]& X_0}
		\endxy
		\]
		In other words, we get a pullback diagram of smooth spaces in which the vertical maps are source maps.
		This implies that $s^*E_G\,\longrightarrow\, E_G$ is a principal $G$-bundle, and the vertical maps are
		morphisms of $G$-bundle.
		Therefore, $X_1$ acts on $E_G$ and $E_G$ becomes an $\mathbb X$-space which actually turns
		${\mathbb E}_G\,=\,[s^*E_G\rightrightarrows E_G]$ into a principal $G$-bundle over the groupoid 
		${\mathbb X}\,=\, [X_1 \rightrightarrows X_0]$. 
		
		In fact, we get a functor.
		$\Psi\,:\, {\mathscr B}un_G({\mathscr X})\,\longrightarrow\,{\mathbb B}un_G({\mathbb X})$,
		which associates the principal $G$-bundle ${\mathbb E}_G\,=\,[s^*E_G\rightrightarrows E_G]$ over
		the Lie groupoid ${\mathbb X}$ to the given principal $G$-bundle $\E_G\,\longrightarrow\,\X$.
		
		Now let us assume conversely, that we have given a principal $G$-bundle ${\mathbb E}_G\,=\,[s^*E_G\rightrightarrows
		E_G]$ over the Lie groupoid ${\mathbb X}\,=\,[X_1 \rightrightarrows X_0]$. This means that we have
		a principal $G$-bundle $\pi\,:\, E_G\,\longrightarrow \,X_0$, which is also a ${\mathbb X}$-space
		satisfying the conditions given in Definition \ref{PrinGGr}. Recall that $$s^*E_G\,=\,
		X_1\times_{s, X_0, \pi} E_G$$ is the pullback along the source map $s$, and ${\mathbb E}_G\,=\,
		[s^*E_G\rightrightarrows E_G]$ is in a natural way a Lie groupoid. It follows also that $s^*E_G$
		becomes in a natural way a $G$-bundle over $X_1$, and we have a morphism of Lie groupoids
		$${\mathbb E}_G\,=\,[s^*E_G\rightrightarrows E_G]\,\longrightarrow\, {\mathbb X}
		$$
		which respects the $G$-bundle structures of $\tau\,:\, s^*E_G\,\longrightarrow\,X_1$ and $\pi\,:\,
		E_G\,\longrightarrow\, X_0$. Let 
		${\mathscr E}_G:={\mathscr B}{\mathbb E}_G$ be the associated differentiable stack of ${\mathbb E}_G$-torsors. The morphism of Lie
		groupoids ${\mathbb E}_G\,\longrightarrow\, {\mathbb X}$ induces a morphism between the associated
		differentiable stacks ${\mathscr E}_G\,\longrightarrow\, {\mathscr X}$, which is representable, as its pullback
		to $X_0$ is the map $\pi\,:\, E_G\,\longrightarrow\, X_0$ between smooth spaces:
		\[
		\xy \xymatrix{E_G\ar[r] \ar[d]& {\mathscr E}_G\ar[d]\\
			X_0\ar[r]& \X}
		\endxy
		\]
		Furthermore, it follows that the pullback of the morphism of stacks ${\mathscr E}_G\,\longrightarrow\,
		{\mathscr X}$ along any morphism $U\,\longrightarrow\,\mathscr X$ from a smooth space $U$ is a principal 
		$G$-bundle. Therefore ${\mathscr E}_G$ is a principal $G$-bundle over the differentiable stack $\X$.
		
		Finally, we get a functor in the opposite direction $\Phi\,:\, {\mathbb B}un_G({\mathbb X})
		\,\longrightarrow\,{\mathscr B}un_G({\mathscr X})$, which associates the principal $G$-bundle $\E_G\,
		\longrightarrow\,\X$ to the given principal $G$-bundle ${\mathbb E}_G\,=\,[s^*E_G\rightrightarrows
		E_G]$ over the Lie groupoid ${\mathbb X}$. It is easy to see that the functors $\Psi$ and $\Phi$
		are mutually inverse and give the desired equivalence of categories. 
	\end{proof}
	
	The following definition allows for yet another alternative description of principal $G$-bundles over 
	Lie groupoids.
	
	\begin{definition}\label{defpbo}
		Let ${\mathbb X}\,=\,[X_1
		\rightrightarrows X_0]$ be a Lie groupoid and $G$ a Lie group. A {\it principal
			$G$-groupoid over} $\mathbb X$ is a Lie groupoid ${\mathbb P}\,=\,[Q\rightrightarrows P]$
		together with a groupoid morphism
		$$
		\xymatrix{G\rto & Q\rto^{\tau} \ar@<-2pt>[d]_{\widetilde{t}} \ar@<+2pt>[d]^{\widetilde{s}}&
			X_1\ar@<-2pt>[d]_t \ar@<+2pt>[d]^s\\
			G\rto & P\rto^{\pi} & X_0}
		$$
		such that both $G\,\longrightarrow\, Q\,\stackrel{\tau}\longrightarrow\, X_1$ and
		$G\,\longrightarrow\, P\,\stackrel{\pi}\longrightarrow\, X_0$ are principal $G$-bundles,
		and, furthermore, the following conditions hold:
		\begin{enumerate}
			\item the source and target maps $\widetilde s$ and $\widetilde t$ on $Q$ are $G$-equivariant,
			
			\item the identity section $\widetilde{e}\, :\, P\, \longrightarrow\, Q$ is $G$-equivariant,
			
			\item the inversion map $\widetilde{i}\, :\, Q\, \longrightarrow\, Q$ satisfies the identity
			$\widetilde{i}(zg) \,=\, \widetilde{i}(z)g^{-1}$ for all $z\, \in\, Q$ and $g\, \in\, G$, and
			
			\item the multiplication map $Q\times_{s,X_0,t} Q\, \longrightarrow\, Q$ is $G$-equivariant
			for the diagonal action of $G$ on $Q\times_{s,X_0,t} Q$.
		\end{enumerate}
	\end{definition}
	
	There is an obvious notion of a morphism between principal $G$-groupoids, and we can speak of the category
	${\mathbb B}un_{G, {\mathbb X}}$ of principal $G$-groupoids over $\mathbb X$, which yields the following equivalent characterization of the category of principal $G$-bundles over a given Lie groupoid $\mathbb X$.
	
	\begin{proposition}
		Let $G$ be a Lie group and ${\mathbb X}\,=\,[X_1\rightrightarrows X_0]$ a Lie
		groupoid. Then there is a canonical equivalence of categories
		$${\mathbb B}un_G({\mathbb X})\,\cong\, {\mathbb B}un_{G, {\mathbb X}}\, ,$$
		where ${\mathbb B}un_G({\mathbb X})$ is the category of principal $G$-bundles
		over $\mathbb X$, and ${\mathbb B}un_{G, {\mathbb X}}$ is the category of principal 
		$G$-groupoids over $\mathbb X$.
	\end{proposition}
	
	\begin{proof}
		The proof for the differentiable category given in \cite[Lemma 2.5]{LTX} and works equally
		well in the complex analytic and algebraic context. 
	\end{proof}
	
	\begin{remark} 
		It can be shown that principal $G$-bundles over a Lie groupoid ${\mathbb
			X}\,=\,[X_1\rightrightarrows X_0]$ are also equivalent to generalized homomorphisms from 
		$\mathbb X$ to the Lie groupoid ${\mathbb B}G\,=\,[G\rightrightarrows *]$, which 
		is Morita equivalent to the {\em gauge groupoid} ${\mathbb P}_{\text{Gauge}}:=[P\times P/G\rightrightarrows X_0]$. In 
		\cite[Prop. 2.13, Prop. 2.14 and Thm 2.15]{LTX} this is discussed in the
		differentiable setting, but this makes sense again in any of the categories $\mathfrak{S}$ of smooth spaces.
	\end{remark}
	
	\section{Connections on Lie groupoids and differentiable stacks}\label{Connections}
	
	In this section, we define and study the notion of connections on Lie groupoids and differentiable stacks. 
	
	\subsection{Connections on Lie groupoids and differentiable stacks} Let us start with the general definition of a connection
	on a Lie groupoid.
	
	Let $(\XX\,=\,[X_1\rightrightarrows X_0]\, , s\, , t\, , m\, , e\, , i)$ be a Lie
	groupoid. Let
	\begin{equation}\label{cK}
	{\mathcal K}\,:=\, \text{kernel}(ds)\, \subset\, TX_1
	\end{equation}
	be the vertical tangent bundle for $s$. Since the source map $s\, :\, X_1\, \longrightarrow\, X_0$ is
	a submersion, ${\mathcal K}$ is a subbundle of $TX_1$. Fix a distribution on $X_1$ given by a subbundle
	${\mathcal H}\, \subset\, TX_1$ such that the natural homomorphism
	\begin{equation}\label{f2}
	{\mathcal K}\oplus {\mathcal H}\, \longrightarrow\, TX_1
	\end{equation}
	is an isomorphism,  so ${\mathcal H}$ is a complement of ${\mathcal K}$. Let
	$$
	d_{\mathcal H}s\, :=\,(ds)\vert_{\mathcal H}\, :\, {\mathcal H}\,\longrightarrow\,
	s^*TX_0
	$$
	be the restriction of $ds$ to ${\mathcal H}$. Note that $d_{\mathcal H}s$
	is an isomorphism because the homomorphism in \eqref{f2} is an isomorphism.
	
	Let
	$
	dt\, :\, TX_1\, \longrightarrow\, t^*TX_0
	$
	be the differential of the map $t$. Consider the homomorphism
	\begin{equation}\label{eq:theta}
	\theta\, :=\,dt\circ (d_{\mathcal H}s)^{-1}\, :\,
	s^*TX_0\, \longrightarrow\, t^*TX_0\, .
	\end{equation}
	For any $y\, \in\, X_1$, let
	\begin{equation}\label{eq:theta2}
	\theta_y\, :=\, \theta\vert_{(s^*TX_0)_y}\, :\, T_{s(y)}X_0\,\longrightarrow\,
	T_{t(y)}X_0
	\end{equation}
	be the restriction of $\theta$ to the fiber $(s^*TX_0)_y\,=\, T_{s(y)}X_0$.
	
	We can then define generally
	
	\begin{definition}[Connection on a Lie groupoid]\label{d-cg}
		A {\em connection} on a Lie groupoid $\XX\,=\,[X_1\rightrightarrows X_0]$ is a distribution
		${\mathcal H}\, \subset\, TX_1$ given by a subbundle complementing $\mathcal K$ such that
		\begin{itemize}
			\item[(i)] for every $x\, \in\, X_0$, the image of the differential
			$$de(x)\,:\, T_x X_0\, \longrightarrow\, T_{e(x)} X_1$$ coincides with
			the subspace ${\mathcal H}_{e(x)}\, \subset\,T_{e(x)} X_1$, and
			\item[(ii)] for every $y\, ,z\, \in\, X_1$ with $t(y)\,=\, s(z)$,
			the homomorphism $$\theta_{m(z,y)}\, :\, T_{s(y)}X_0\,\longrightarrow\,
			T_{t(z)}X_0$$ coincides with the composition $\theta_z\circ\theta_y$ (see
			\eqref{eq:theta2}).
		\end{itemize}
		A connection on $\XX$ is said to be {\em flat} (or {\em integrable}) if the distribution ${\mathcal H}\,\subset\, TX_1$ defining
		the connection is integrable.
	\end{definition}
	Henceforth we will denote  $m(y, z)=y\circ z.$ Similarly for the pairs $(y, v), (z, w)\in T X_1$ such that $s(y)=t(z), dt_y(v)=ds_z(w),$ the composition $dm\bigl((y, v), (z, w)\bigr)$ in the induced tangent Lie groupoid  $T\XX=[TX_1\rightrightarrows TX_0]$ will be denoted as $(y\circ z, v\circ w).$
	
	We will now discuss some functoriality properties for our notion of connection on Lie groupoids. 
	
	\begin{proposition}\label{prop-ic}
		Let ${\mathbb X}\,=\,[X_1\rightrightarrows X_0]$ and ${\mathbb
			Y}\,=\,[Y_1\rightrightarrows Y_0]$ be Lie groupoids, and let
		$$
		\phi\,=\, (\phi_1,\, \phi_0)\,:\, {\mathbb X}\,\longrightarrow\, {\mathbb Y}
		$$
		be a Morita morphism. Let
		$$
		{\mathcal H}\, \subset\, TY_1
		$$
		be a distribution that defines a connection on the groupoid ${\mathbb Y}$. Then
		$\mathcal H$ induces a connection on the groupoid ${\mathbb X}$.
	\end{proposition}

	\begin{proof}
		Recall from Definition \ref{dme}(i) that the map $\phi_0\,:\, X_0\,\longrightarrow\, Y_0$ is a
		surjective submersion. Hence from Definition \ref{dme}(ii) it follows that
		$$\phi_1\,:\, X_1\,\longrightarrow\, Y_1$$
		is also a surjective submersion. Let
		$$
		d\phi_1\,:\, TX_1\,\longrightarrow\, TY_1
		$$
		be the differential of the map $\phi_1$.  Let ${\mathcal H}\, \subset\, TY_1$ be the distribution inducing a connection on the groupoid ${\mathbb Y}$. 
		Now define
		$$
		\widetilde{\mathcal H}\, :=\, (d\phi_1)^{-1}(\phi^*_1 {\mathcal H})\, \subset\, TX_1\, .
		$$
		Since $\phi_1$ is a surjective submersion, it follows that $\widetilde{\mathcal H}\, \subset\, TX_1$
		is a distribution and therefore defines a connection on the groupoid ${\mathbb X}$.
	\end{proof}
	
	Similarly, as in \cite{B1}, we can study functorial behaviour with respect to horizontal morphisms of Lie groupoids.
	\begin{definition} 
		Let ${\mathbb X}\,=\,[X_1\rightrightarrows X_0]$ and ${\mathbb Y}\,=\,[Y_1\rightrightarrows Y_0]$ be Lie groupoids and 
		$\phi\,=\, (\phi_1,\, \phi_0)\,:\, {\mathbb X}\,\longrightarrow\, {\mathbb Y}$ be a morphism. Suppose that ${\mathcal H}\subset TX_1$ defines a connection on $\XX$ and ${\mathcal L}\subset TY_1$ a connection on $\YY$. The morphism $\phi$ is called {\em horizontal} if the differential 
		$d\phi_1\,:\, TX_1\,\longrightarrow\, TY_1$ maps ${\mathcal H}$ into ${\mathcal L}$.
	\end{definition}

	\begin{proposition}
	Let $\XX$ and $\YY$ be Lie groupoids together with an \'etale morphism $\phi\,=\, (\phi_1,\, \phi_0)\,:\, {\mathbb X}\,\longrightarrow\, {\mathbb Y}$. Then any connection ${\mathcal L}$ on $\YY$ induces a unique connection ${\mathcal H}$ on $\XX$  such that $\phi$ is a horizontal morphism. If the connection ${\mathcal L}$ is integrable, then also the induced connection ${\mathcal H}$ is integrable. 
	\end{proposition}
	\begin{proof}
	 Given a connection ${\mathcal L}$ on the groupoid $\YY$, the unique connection ${\mathcal H}$ on the groupoid $\XX$ is given by the fiber product ${\mathcal H}:={\mathcal L}\times_{TY_1}TX_1.$	
	\end{proof}
	
	In particular, a connection (respectively, flat connection) on an \'etale groupoid induces then a connection (respectively, flat connection) on the associated Deligne-Mumford stack ${\mathscr B}{\mathbb X}$, the classifying stack of $\mathbb X$-torsors. 
	
\begin{example}\label{Example:connectionOnMM}
		Let $\mathbb{X}=[X\rightrightarrows X]$ be the Lie groupoid associated to a smooth space $X$. Then $\text{kernel}(ds)\,=\,\{0\}$ and
		the map $x\,\longmapsto \,\mathcal{H}_xX\,:=\,T_xM$, $x\, \in\, X$, defines an integrable connection $\mathcal{H}_{\mathbb{X}}\subset 
		TX$ on $\mathbb{X}$.
	\end{example}	
\begin{example}
More generally, for an \'etale Lie groupoid $\mathbb{X}\,=\,[X_1\rra X_0],$ since  the differential $ds\,\colon\, TX_1\,\longrightarrow\, s^*TX_0$ is an 
		isomorphism,   the distribution $\mc{H}_{\mb{X}}=\, TX_1$ is an integrable connection on the Lie 
		groupoid $\mb{X}$. 
\end{example}

\begin{example}\label{Example:Vectorbundlegroupoid}
		Let $\pi\,\colon\, E\,\longrightarrow\, X$ be a finite rank vector bundle over a smooth space $X.$ We get a Lie groupoid $\mb{X}=[E\rra X]$ with  the source and target maps  both being $\pi.$ Then any two composable morphisms belong to the same fibre,  and the composition can be defined as addition of vectors in that fibre. Any connection 
		on the vector bundle $\pi\,\colon\, E\,\longrightarrow\, X$ smoothly splits $E$ into the horizontal component and  $\text{kernel}(ds).$ This gives a connection on the associated  Lie groupoid $\mb{X}=[E\rra X]$. Integrability of one connection implies integrability of the other.	
			\end{example}

\begin{example}\label{Example:Gaugegrouoidconnection}
Let $G$ be a Lie group. Given a principal $G$-bundle $\pi\,\colon\,P\,\longrightarrow\, X$ over a smooth space $X$, 
one defines the {\em Atiyah} or {\em gauge groupoid} $\mb{P}_{\rm Gauge}:=\,[\frac{P\times P}{G}\rra X\,]$ by 
building the quotient of the groupoid $[P\times P\rra P]$ with respect to the diagonal action of $G$ on $P\times P.$ 
A connection $\omega$ on the principal bundle $\pi\,\colon\, P\,\longrightarrow\, X$ gives a $G$-invariant horizontal 
distribution $\mc{H}\subset TP$ complementing the $\text{kernel}.$ We define a connection $\mc{H}_{\mb{P}_{\rm 
Gauge}}$ on the Lie groupoid $\mb{P}_{\rm Gauge}$ by $${\mc{H}}_{[p, q]}\,\,:=\,\,\frac{{\mc{H}}_{p}\oplus T_qP}{T_{p, 
q}\,{(p, q)\cdot G}}\,\,\subset\, \,T_{[p, q]}\frac{P\times P}{G}\, ,$$ where $(p,\, q)\cdot G\,\subset\, P\times P$ is 
the orbit of the element $(p,\, q)\,\in\, P\times P.$ Integrability of the connection $\omega$ implies integrability 
of $\mc{H}_{\mb{P}_{\rm Gauge}}.$
\end{example}
	
	\begin{remark}\label{re:otherconnections}
		Related constructions of connections and flat connections on groupoids and on stacks in the algebro-geometric, differentiable and holomorphic context were also studied by Behrend in \cite{B1}. Flat connections in the differentiable setting for Lie groupoids were also independently introduced as \'etalifications by Tang \cite{T}. These constructions give all rise to subgroupoids of the associated tangent groupoid $T\XX$ of a groupoid $X$, which in the differentiable category
		is equivalent to the horizontal paths forming a subgroupoid of the path groupoid of $\XX$. This is used
		by Laurent-Gengoux, Sti\'enon and Xu in \cite{LSX} to define connections in the general framework of non-abelian differentiable gerbes via Ehresmann connections on Lie groupoid extensions. They also appear as multiplicative distributions on Lie groupoids in recent work by Drummond, and Egea \cite{DE} and Trentinaglia \cite{Tr}.
		Another definition of a general Ehresmann connection for Lie groupoids was more recently also given and discussed by Arias Abad and Crainic
		\cite{AC}. 
	\end{remark}
	
	We will now point out some additional properties of our constructions of connections for Lie groupoids to highlight the relation with other constructions existing in the literature and here in particular with those in \cite{B1, AC}.

	Let ${\mathcal H}\subset TX_1$ be a connection on a Lie groupoid $\XX=[X_0\rightrightarrows X_1]$.  The connection defines a groupoid  $[s^*TX_0\rightrightarrows TX_0]$  as follows. Without loss of any information we denote an element $(\gamma, s(\gamma), v)$ of $s^*TX_0$ by $(\gamma, v)$, where $v\in T_{s(\gamma)}X_0$. The source, target and composition maps are then respectively given by 
	\begin{equation}\label{eq:grpds*tx}
	\begin{split}
	&s\,\colon\,(\gamma,\, v)\,\longmapsto\, (s(\gamma),\, v), \qquad \gamma\,\in\, X_1,\, v\,\in \,T_{s(\gamma)}X_0,\\
	&t\,\colon\,(\gamma,\, v)\,\longmapsto\, (t(\gamma),\, \theta_{\gamma}(v)),\\
	&(\gamma_2,\, v_2)\circ (\gamma_1,\, v_1)\,=\,(\gamma_2\circ \gamma_1,\, v_1).	\\
	\end{split}
	\end{equation}
	The inversion and unit maps are obvious. Note that the composition is well defined because of condition (ii) in Definition~\ref{d-cg}. Similarly $[t^*TX_0\rightrightarrows TX_0]$ is a groupoid with the corresponding structure maps given as:
	\begin{equation}\label{eq:grpdt*Tx}
	\begin{split}
	&s\,\colon\,(\gamma, \,u)\,\longmapsto\, (s(\gamma), \,\theta_{{\gamma}^{-1}}(u)), \qquad \gamma\,\in
	\, X_1,\, u\,\in\, T_{t(\gamma)}X_0,\\
	&t\,\colon\,(\gamma,\, u)\,\longmapsto\, (t(\gamma),\, u),\\
	&(\gamma_2,\, u_2)\circ (\gamma_1,\, u_1)\,=\,(\gamma_2\circ \gamma_1,\, u_2).	\\
	\end{split}
	\end{equation}
	In fact $$\theta\,\colon\, (\gamma, \,v)\,\longmapsto\, (\gamma, \,
	{d_{\mathcal H}s}_{{\gamma}}^{-1}(v))\,\longmapsto\, (\gamma, \,\underbrace{dt_{\gamma}\circ {d_{\mathcal H}s}_{{\gamma}}^{-1}}_{\theta_{\gamma}} (v))$$ in \eqref{eq:theta} defines an isomorphism of Lie groupoids
	\begin{equation}\label{isosourcetarhgetgrpd}
	\xymatrix{ s^*TX_0\rto^{\theta} \ar@<-2pt>[d] \ar@<+2pt>[d]&
		t^*TX_0\ar@<-2pt>[d] \ar@<+2pt>[d]\\
		TX_0\rto^{{\rm Id}} & TX_0}
	\end{equation}
	Moreover, the condition (ii) in Definition~\ref{d-cg} implies that
	\begin{equation}\label{thetads}
	{d_{\mathcal H}s}_{{\gamma_2\circ \gamma_1}}^{-1}(v)-{d_{\mathcal H}s}_{{\gamma_2}}^{-1}(\theta_{\gamma_1}(v))\circ 
	{d_{\mathcal H}s}_{{\gamma_1}}^{-1}(v)\,\in\, {\rm ker}(dt_{\gamma_2\circ \gamma_1})
	\end{equation} 
	for any pair of composable $\gamma_1,\, \gamma_2\,\in\, X_1$ and $v\,\in\, T_{s(\gamma_1)}X_0.$ We denote the
	corresponding element in ${\rm ker}(dt_{\gamma_2\circ \gamma_1})$ by
	$${\mathfrak K}(\gamma_2,\, \gamma_1,\, v)\,:=\,
	{d_{\mathcal H}s}_{{\gamma_2\circ \gamma_1}}^{-1}(v)-{d_{\mathcal H}s}_{{\gamma_2}}^{-1}(\theta_{\gamma_1}(v))\circ 
	{d_{\mathcal H}s}_{{\gamma_1}}^{-1}(v).
	$$
	
	The outcome of the above observation is the following lemma.
	
	\begin{lemma}
		$({d_{\mathcal H}s}^{-1},\, {\rm{Id}})\,\colon\, [s^*{TX_0}\rightrightarrows TX_0]\,\longrightarrow\,
		[TX_1\rightrightarrows TX_0]$ defines an essentially surjective, faithful functor if and only if
		${\mathfrak K}(\gamma_2, \,\gamma_1,\, v)$ vanishes for all pairs of composable arrows $\gamma_1,\,\gamma_2\,\in\, X_1$
		and $v\,\in \,T_{s(\gamma_1)}X_0.$ 
	\end{lemma}
	
	\begin{proof}
		That it is essentially surjective is evident. Now we have:
		\begin{equation*}
		{d_{\mathcal H}s}^{-1}\big((\gamma_2, v_2)\circ (\gamma_1, v_1)\big)
		={d_{\mathcal H}s}^{-1}\big((\gamma_2\circ \gamma_1, v_1)\big)
		=(\gamma_2\circ \gamma_1, {d_{\mathcal H}s}_{\gamma_2\circ \gamma_1}^{-1}(v_1)).
	         \end{equation*}
		
		On the other hand, we get:
		\begin{eqnarray*}
		{d_{\mathcal H}s}^{-1}\big((\gamma_2, v_2)\big)\circ {d_{\mathcal H}s}^{-1}\big((\gamma_1, v_1)\big) = \big(\gamma_2, {d_{\mathcal H}s}_{\gamma_2}^{-1}(v_2)\big)\circ \big(\gamma_1, {d_{\mathcal H}s}_{\gamma_1}^{-1}(v_1)\big)\\
		=  \big(\gamma_2\circ \gamma_1, {d_{\mathcal H}s}_{\gamma_2}^{-1}(v_2)\circ {d_{\mathcal H}s}_{\gamma_1}^{-1}(v_1)\big).
		\end{eqnarray*}
				
		Then functoriality is now immediate from the vanishing of the left hand side in \eqref{thetads}.
		Since ${d_{\mathcal H}s}^{-1}$ is injective, the functor is also faithful. 
	\end{proof}
	
	\begin{remark}
	From the above we see that ${\mathfrak K}(\gamma_2,\, \gamma_1,\, v)$ gives an obstruction for the groupoid $[{\mathcal H}
	\,\rightrightarrows \,TX_0]$  to be a subgroupoid of the tangent groupoid $T\XX\,=\,[TX_1\rightrightarrows TX_0]$. 
	Indeed vanishing of the element in \eqref{thetads} is equivalent to the necessary and sufficient condition for
	a splitting as in \eqref{f2} to yield a subgroupoid, mentioned in Lemma 2.13 of \cite{AC}. Moreover
	\eqref{isosourcetarhgetgrpd} implies that both diagrams 
	\begin{equation}\nonumber
	\xy \xymatrix{ {\mathcal H} \ar@<-.5ex>[r]\ar@<.5ex>[r]  \ar[d]_{}&
		TX_0\ar[d]^{}\\
		X_1\ar@<-.5ex>[r]\ar@<.5ex>[r]  & X_0}
	\endxy
	\end{equation}	
	are Cartesian and thus we arrive at the definition of a connection on a Lie groupoid as given in \cite[Def. 2.1]{B1}. In conclusion, we see that our definition of a connection on a Lie groupoid is more general than the one in 
	\cite{B1}, but more strict than the definition in \cite{AC}. In fact, a partition of unity argument shows that any Lie groupoid admits a connection in the sense 
	of \cite{AC}. If such a connection satisfies in addition the conditions (i) and (ii) of Definition~\ref{d-cg}, then we obtain a connection as defined in this article.
	\end{remark}

	\subsection{Making $TX_0$ a vector bundle over $\XX$}\label{se3.2}
	
	We will now state another useful geometric interpretation. Let $\mathcal H$ be a connection on a given Lie groupoid $\XX \,=\,[X_1
	\rightrightarrows X_0]$, and let
	$$
	\tau\, :\, TX_0\, \longrightarrow\, X_0
	$$
	be the natural projection. Consider the fiber product
	$$
	Y_1\, :=\, X_1\times_{s,X_0,\tau} TX_0\, .
	$$
	Let
	$s'\, :\, Y_1\, \longrightarrow\, TX_0$
	be the projection to the second factor. Define the morphism
	$$
	t'\, :\, Y_1\, \longrightarrow\, TX_0\, , \ ~ 
	(x,\, v)\,\longmapsto\, \theta_x(v)\,\in\, T_{t(x)}X_0\, ,
	$$
	where $\theta_x$ is constructed as in (\ref{eq:theta2}). Furthermore, let
	$$
	p\, :\, Y_1\, =\, X_1\times_{s,X_0,\tau} TX_0\, \longrightarrow\, X_1
	$$
	be the projection to the first factor. For
	any $z,\, y\, \in\, Y_1$ with $t'(y)\,=\, s'(z)$, define
	$$
	m'(z,\, y)\,:=\, (m(p(z),\,p(y)),\,  s'(y))\,.
	$$
	Note that $s(p(z))\,=\, t(p(y))$, so $m(p(z),\,p(y))$ is defined. Let $e'$ be the morphism
	$$
	e'\, :\, TX_0\, \longrightarrow\, Y_1\, ,\ ~
	v\,\longmapsto\, (e(\tau(v)),\, v).
	$$
	Finally, let
	$$
	i'\, :\, Y_1\, \longrightarrow\, Y_1,\,  \ ~
	(z,\, v)\, \longmapsto\, (i(z),\, \theta_z(v))
	$$
	be the involution. It is straightforward to check that $([Y_1\rightrightarrows TX_0],\,  s',\, t',\, m',
	\, e',\, i')$ is a Lie groupoid. In other words, $TX_0$ is a vector bundle over the groupoid $\XX$.
	
	\subsection{Characteristic differential forms for connections on $\XX$}\label{se3.3}
	We will now study the behaviour and interpretation of connections on Lie groupoids in terms of differential forms. Let $\mathcal H$ be a connection on the Lie groupoid $\XX=[X_1\rightrightarrows X_0]$. Using the canonical decomposition of the tangent space
	$$
	TX_1\,=\, {\mathcal H}\oplus \text{kernel}(ds)\, ,
	$$
	we get a projection
	$$
	\wedge^j TX_1\,\longrightarrow\, \wedge^j{\mathcal H}
	\,\hookrightarrow\, \wedge^j TX_1\, .
	$$
	The composition $\wedge^j TX_1\,\longrightarrow\,
	\wedge^j TX_1$ gives by duality an endomorphism of the space of
	$j$-forms on $X_1$.
	
	For a differential form $\omega$ on $X_1$, the differential form on $X_1$
	{\em induced} by the above endomorphism for the given connection $\mathcal H$ will be denoted by $H(\omega)$.
	
	\begin{definition}
		A {\em differential $j$-form} on the Lie groupoid $[X_1\rightrightarrows X_0]$ is a differential $j$-form
		$\omega$ on $X_0$ such that $\mc{H}(s^*\omega)\,=\, \mc{H}(t^*\omega)$.
	\end{definition}
	
	Induced differential forms for an integrable distribution observe the following basic property.
	
	\begin{proposition}\label{lem1}
		Let $\mathcal H$ be a connection on a Lie groupoid $\XX=[X_1\rightrightarrows X_0]$ and assume that the distribution ${\mathcal H}\, \subset\, TX_1$ is integrable. Let $\omega$ be a differential form on $X_0$ such that
		$\mc{H}(s^*\omega)\,=\, {\mc H}(t^*\omega)$. Then $\mc{H}(s^*d\omega)\,=\, \mc{H}(t^*d\omega)$.
	\end{proposition}
	
	\begin{proof}
		Take a point $x\,\in\, X_1$. Let ${\mathcal L}$ be the locally defined
		leaf of ${\mathcal H}$ passing through $x$. Let
		$$
		\iota\, :\, {\mathcal L}\, \hookrightarrow\, X_1
		$$
		be the inclusion map. Since $\mc{H}(s^*\omega)\,=\, \mc{H}(t^*\omega)$, it follows
		immediately that $\iota^*s^*\omega\,=\, \iota^*t^*\omega$. Therefore we have,
		$$
		\iota^*s^*(d\omega)\,=\, d \iota^*s^*\omega\,=\,d \iota^*t^*\omega
		\,=\,\iota^*t^*(d\omega)\, .
		$$
		But this implies that $\mc{H}(s^*(d\omega))(x)\,=\, \mc{H}(t^*(d\omega))(x)$ and so
		we conclude that $\mc{H}(s^*(d\omega))\,=\, \mc{H}(t^*(d\omega))$, which finishes the proof.
	\end{proof}

	\section{Connections on principal bundles over Lie groupoids and differentiable stacks}
	\label{Bundles}
	
	We will now study in detail the notion and interplay of connections on principal $G$-bundles over Lie groupoids and differentiable stacks. 
	
	\subsection{Connections on principal $G$-bundles over Lie groupoids}
	
	Let us start with the groupoid picture. Let $(\XX\,=\,[X_1\rightrightarrows X_0],\,  s,\,  t,\,  m,\,
	e,\,  i)$ be a Lie groupoid and $G$ a Lie group. The Lie algebra of $G$ will be denoted by $\mathfrak g$. We shall consider $\mathfrak g$ as a $G$-module using the adjoint action. Let
	\begin{equation}\label{alpha}
	\alpha\, :\, E_G\, \longrightarrow\, X_0
	\end{equation}
	be a principal $G$-bundle over $X_0$. The adjoint vector bundle for $E_G$
	$$
	\text{ad}(E_G)\, :=\, E_G\times^G{\mathfrak g}\, \longrightarrow\, X_0
	$$
	is the bundle associated to $E_G$ for the adjoint action of $G$ on $\mathfrak g$. The action of
	$G$ on $E_G$ induces an action of $G$ on the direct image $\alpha_*TE_G$, where $\alpha$ is the
	projection as given in \eqref{alpha}. 
	
	\begin{definition}
		The {\em Atiyah bundle} for a principal $G$-bundle $E_G$ over $X_0$ is the invariant direct image
		\begin{equation}\label{eq:a-e2}
		\text{At}(E_G)\,:=\, (\alpha_*TE_G)^G\, \subset\, \alpha_*TE_G\, .
		\end{equation}
	\end{definition}
	
	Therefore, we have $\text{At}(E_G)\,=\, (TE_G)/G$ (see also \cite{At} for the classical and analogue notion of the 
	Atiyah bundle in the complex analytic context).
	
	Let
	\begin{equation}\label{qE}
	q_E\, :\, TE_G\, \longrightarrow\, (TE_G)/G\,=\, \text{At}(E_G)
	\end{equation}
	be the quotient map.
	
	Let $T_{\rm rel} \, \subset\, TE_G$ be the relative tangent bundle for the projection
	$\alpha$ in \eqref{alpha}. It fits into the short exact sequence of vector bundles
	\begin{equation}\label{eq:a-e3}
	0\, \longrightarrow\, T_{\rm rel}\, \stackrel{\iota_0}{\longrightarrow}\, TE_G\,
	\stackrel{d\alpha}{\longrightarrow}\, \alpha^* TX_0 \, \longrightarrow\, 0\, ,
	\end{equation}
	where $d\alpha$ is the differential of $\alpha$. Using the action of $G$ on $E_G$, the vector bundle
	$T_{\rm rel} \,\longrightarrow\, E_G$ gets identified with the trivial vector bundle
	$E_G\times\mathfrak g\,\longrightarrow\, E_G$. Therefore, we see that
	$$
	\text{ad}(E_G)\,=\, (\alpha_*T_{\rm rel})^G\, ,
	$$
	and so $\text{ad}(E_G)\,=\,(\alpha_*T_{\rm rel})/G$. The map of quotients
	$T_{\rm rel}/G\, \longrightarrow\, (TE_G)/G$ induced by the inclusion
	$T_{\rm rel} \,\hookrightarrow\, TE_G$ makes 
	$\text{ad}(E_G)$ a subbundle of $\text{At}(E_G)$.
	The differential of $d\alpha$ in \eqref{eq:a-e3} being $G$-equivariant therefore produces a homomorphism
	\begin{equation}\label{eda}
	(d\alpha)'\, :\, \text{At}(E_G)\,\longrightarrow\, TX_0\, .
	\end{equation}
	Combining these we obtain the {\em Atiyah exact sequence} (see \cite{At}):
	\begin{equation}\label{eq:a-e-s}
	0\,\longrightarrow\, \text{ad}(E_G)\,\longrightarrow\,\text{At}(E_G)\,
	\stackrel{(d\alpha)'}{\longrightarrow}\, TX_0 \,\longrightarrow\, 0\, .
	\end{equation}
	This exact sequence is the quotient of the exact sequence in \eqref{eq:a-e3} by the
	action of $G$. Now we can define the general notion of a connection on a principal 
	$G$-bundle (compare \cite{At})
	
	\begin{definition}
		A {\em connection} on a principal $G$-bundle $E_G$ over $X_0$ is a splitting of the Atiyah exact sequence
		\begin{equation}
		0\,\longrightarrow\, \text{ad}(E_G)\,\longrightarrow\,\text{At}(E_G)\,
		\stackrel{(d\alpha)'}{\longrightarrow}\, TX_0 \,\longrightarrow\, 0\, .
		\end{equation}
	\end{definition}
	
	Giving a splitting of the exact sequence in \eqref{eq:a-e-s} is evidently equivalent to
	giving a $G$-equivariant splitting of the exact sequence in \eqref{eq:a-e3}. Therefore,
	a connection on $E_G$ is a homomorphism
	$$
	D_{E_G}\, :\, TE_G \,\longrightarrow\, T_{\rm rel}
	$$
	such that
	\begin{itemize}
		\item $D_{E_G}$ is $G$-equivariant, and
		
		\item $D_{E_G}\circ\iota_0\,=\, \text{Id}_{T_{\rm rel}}$, where $\iota_0$ is the
		homomorphism in \eqref{eq:a-e3}.
	\end{itemize}
	Recall that $T_{\rm rel}\,=\, E_G\times\mathfrak g$. Therefore, any homomorphism
	$D_{E_G}\, :\, TE_G \,\longrightarrow\, T_{\rm rel}$ satisfying the above two conditions is
	a $\mathfrak g$-valued $1$-form $\widehat{D}_{E_G}$ on $E_G$ such that the corresponding homomorphism
	\begin{equation}\label{cf}
	\widehat{D}_{E_G}\, :\, TE_G\,\longrightarrow\, \mathfrak g
	\end{equation}
	is $G$-equivariant for the adjoint action of $G$ on $\mathfrak g$.
	
	Now we equip $E_G$ with the structure of a principal $G$-bundle on the fixed Lie groupoid $\XX\,=\,
	[X_1\rightrightarrows X_0]$. First we shall use the description of a principal $G$-bundle over
	$\XX$ given in Definition \ref{defpbo}. So we have a
	Lie groupoid ${\mathbb E}_G\,=\,[Q\rightrightarrows P]$
	together with a groupoid morphism
	\begin{equation}\label{gpm}
	\xymatrix{G\rto & Q\rto^{\tau} \ar@<-2pt>[d]_{\widetilde{t}} \ar@<+2pt>[d]^{\widetilde{s}}&
		X_1\ar@<-2pt>[d]_t \ar@<+2pt>[d]^s\\
		G\rto & E_G\rto^{\pi} & X_0}
	\end{equation}
	such that both $Q\,\stackrel{\tau}\longrightarrow\, X_1$ and 
	$E_G\,\stackrel{\pi}\longrightarrow\, X_0$ are principal $G$-bundles, and the four conditions in 
	Definition \ref{defpbo} are satisfied.
	
	Let $\nabla$ be a connection on the principal $G$-bundle $E_G\,\longrightarrow\, X_0$. Let
	\begin{equation}\label{cf2}
	\widehat{\nabla}\, :\, TE_G\,\longrightarrow\, \mathfrak g
	\end{equation}
	be the $G$-equivariant $\mathfrak g$-valued $1$-form on $E_G$ as in \eqref{cf} corresponding to $\nabla$.
	
	\begin{definition}[Connection on a principal bundles over a Lie groupoid]\label{dco0}
		A {\em connection} on a principal $G$-bundle ${\mathbb E}_G\,=\,[Q\rightrightarrows 
		E_G]$ over the Lie groupoid $\XX\,=\,[X_1\rightrightarrows X_0]$ is a connection $\nabla$ 
		on the principal $G$-bundle $E_G\, \longrightarrow\, X_0$ such that the two $\mathfrak g$-valued $1$-forms
		$\widetilde{s}^* \widehat{\nabla}$ and $\widetilde{t}^* \widehat{\nabla}$ on $Q$ coincide, where
		$\widehat{\nabla}$ is the $1$-form $\widehat{\nabla}\, :\, TE_G\,\longrightarrow\, \mathfrak g$ associated to
		the connection $\nabla$.
	\end{definition}
	
	Now we adopt Definition \ref{PrinGGr}. Let
	$${\mathbb E}_G\,:=\,[s^*E_G\rightrightarrows E_G]$$
	be a principal $G$-bundle over $\mathbb X$ as in 
	Definition \ref{PrinGGr}. Let $\widehat{s}\, :\, s^*E_G\,\longrightarrow\, E_G$
	and $\widehat{t}\, :\, s^*E_G\,\longrightarrow\, E_G$ respectively be the source map and the target map.
	
	The following definition is then evidently equivalent to Definition \ref{dco0}.

	\begin{definition}\label{dco}
		A {\em connection} on the principal $G$-bundle ${\mathbb E}_G\,=\,[s^*E_G\rightrightarrows E_G]$
		over the Lie groupoid $\XX\,=\,[X_1\rightrightarrows X_0]$ is a connection $\nabla$ on the principal
		$G$-bundle $E_G\, \longrightarrow\, X_0$ such that the two $\mathfrak g$-valued $1$-forms
		$\widehat{s}^* \widehat{\nabla}$ and $\widehat{t}^* \widehat{\nabla}$ on $s^*E_G$ coincide, where
		$\widehat{\nabla}$ is the $1$-form $\widehat{\nabla}\, :\, TE_G\,\longrightarrow\, \mathfrak g$ associated to
		the connection $\nabla$.
	\end{definition}
	
	\begin{definition}
		Let $\widehat\nabla$ be a connection on a principal $G$-bundle ${\mathbb E}_G\,=\,[s^*E_G\rightrightarrows
		E_G]$ over the Lie groupoid $\XX\,=\,[X_1\rightrightarrows X_0]$ given by a connection $\nabla$ on $E_G$.
		The {\it curvature} of $\widehat\nabla$ is defined to be the curvature of
		$\nabla$. In particular, the connection $\widehat\nabla$ is called \textit{flat} if $\nabla$ is integrable.
	\end{definition}
	
	\begin{proposition}\label{cpb}
		Let ${\mathbb X}\,=\,[X_1\rightrightarrows X_0]$ and ${\mathbb
			Y}\,=\,[Y_1\rightrightarrows Y_0]$ be Lie groupoids, and let 
		$\phi\,:\, {\mathbb X}\,\longrightarrow\, {\mathbb Y}$ be a 
		morphism given by maps
		$$\phi_0\,:\, X_0\,\longrightarrow\, Y_0\, , \ \ \phi_1\,:\, X_1\,\longrightarrow\, Y_1$$
		(see Definition \ref{dme}). Let 
		$${\mathbb E}_G\,:=\,[s^*E_G\rightrightarrows E_G]$$
		be a principal $G$-bundle on $\mathbb Y$ equipped with a connection 
		$\widehat\nabla$ given by a connection $\nabla$ on the principal $G$-bundle
		$E_G\, \longrightarrow\, Y_0$. Then the pulled back connection $\phi^*_0\nabla$
		on $\phi^*E_G\, \longrightarrow\, X_0$ is a connection on the principal $G$-bundle
		$\phi^*{\mathbb E}_G$ over the Lie groupoid $\mathbb X$.
	\end{proposition}
	
	\begin{proof}
		Let $s$ and $s'$ (respectively, $t$ and $t'$) be the source maps (respectively,
		target maps) of ${\mathbb X}$ and ${\mathbb Y}$ respectively.
		The two maps $s'\circ\phi_1$ and $\phi_0\circ s$ from $X_1$ to $Y_0$ coincide.
		Similarly, the two maps $t'\circ\phi_1$ and $\phi_0\circ t$ from $X_1$ to $Y_0$ also coincide.
		
		Using this, it is straightforward to check that the connection $\phi^*_0\nabla$
		on $\phi^*E_G\, \longrightarrow\, X_0$ is a connection on the principal $G$-bundle
		$\phi^*{\mathbb E}_G$ over the groupoid $\mathbb X$.
	\end{proof}
	
	\begin{example}
	A connection on a principal $G$-bundle  ${\mathbb E}_G\,=\,[Q\rightrightarrows 
		E_G]$ over the Lie groupoid $\XX\,=\,[X\rightrightarrows X]$ is the same as 
		a connection on the (ordinary) principal $G$-bundle $E_G\rightarrow X$ over the smooth space $X$. 
	\end{example}
	\begin{example}
	Let $\pi_0\colon\,E\rightarrow X_0$ be a vector bundle and $\pi_G\colon E_G\rightarrow X_0$ a principal $G$-bundle over  $X_0.$ Let $\mb{X}:=[E\rra X_0]$ be the 	Lie groupoid introduced in Example~\ref{Example:Vectorbundlegroupoid}. Then  $E_G$ is an $\mb{X}$-space with respect to  the map $(v, p)\mapsto p,$ for all $(v, p)\in E\times E_G$ satisfying  $\pi_0(v)=s(v)=\pi_G(p).$ The action of $\mb{X}$ is obviously compatible with the action of $G$ on $E_G,$ and thus we obtain a principal $G$-bundle over $\mb{X}.$ Then any connection on the (ordinary) principal $G$-bundle $E_G\rightarrow X_0$ defines a connection on the principal $G$-bundle over $\mb{X}.$	
	
	\end{example}
	
	\begin{example}
	Let $E_G\rightarrow X_0$ be  an $H$-equivariant principal $G$-bundle over $X_0.$	Let ${\mathbb X}\,=\,[H\times X_0\rightrightarrows
		X_0]$ be the transformation groupoid. As in Example~\ref{Ex:G-Hequi}, consider $E_G\rightarrow X_0$ to be a $G$-bundle over ${\mathbb X}.$ 
		Then a connection $1$-form on the (ordinary) principal $G$-bundle over $X_0$	 is a connection on 	the $G$-bundle over $\mb{X}$ if and only if it is 
		$H$-invariant and vanishes on the fundamental vector field generated by the action of $H$ on $E_G.$
	
	\end{example}
	
\begin{remark} The definition of a connection on a principal $G$-bundle over a Lie groupoid given in this article is less rigid then the one by Laurent-Gengoux, Tu and Xu \cite[Def. 3.5]{LTX}. the difference between our approach and the one in \cite{LTX} becomes clear when comparing the associated de Rham complexes and Chern-Weil maps, which is discussed in our follow-up article \cite[Sect.~5 \& 6]{BCKN2}. In particular, it is not hard to see that the two definitions coincide in the case of principal $G$-bundles over an \'etale Lie groupoid $\XX\,=\,[X_1 \rightrightarrows X_0]$ with integrable distribution ${\mathcal H}\,=\,TX_1$.
\end{remark}

The space of connections on a principal $G$-bundle ${\mathbb E}_G$ over the Lie groupoid $\XX$ is an
affine space for the space of all $\text{ad}(E_G)$-valued $1$-forms on the groupoid.
	
	Henceforth, given a Lie groupoid $\XX\,=\,[X_1 \rightrightarrows X_0]$ we will assume that there exist a given 
	integrable distribution ${\mathcal H}\, \subset\, TX_1$, in other words, we have given a {\em flat} connection on 
	the Lie groupoid $\XX$.
	
	Consider now again the Atiyah exact sequence as constructed in (\ref{eq:a-e-s}). The Lie bracket
	of vector fields defines a Lie algebra structure on the sheaves
	of sections of all three vector bundles.
	The Lie algebra structure on the sheaf of sections of $\text{ad}(E_G)$
	is linear with respect to the multiplication by functions on $X_0$, or
	in other words, the fibers of $\text{ad}(E_G)$ are Lie algebras.
	
	Let $\mathfrak g$ be again the Lie algebra for the Lie group $G$.
	Recall that $\text{ad}(E_G)\,=\,(E_G\times {\mathfrak g})/G$ is the vector bundle on $X_0$ associated to
	principal $G$-bundle $E_G$ for the adjoint action of $G$ on $\mathfrak g$. 
	Since this adjoint action of $G$ preserves the Lie
	algebra structure of $\mathfrak g$, it follows that the fibers
	of $\text{ad}(E_G)$ are Lie algebras identified with $\mathfrak g$ up to
	conjugations.
	
	Given any splitting of the Atiyah exact sequence (\ref{eq:a-e-s})
	$$
	D\, :\, TX_0\, \longrightarrow\,\text{At}(E_G)\, ,
	$$
	the obstruction for $D$ to be compatible with the Lie algebra structure
	is given by a section
	$$
	K(D)\, \in\, H^0(X_0,\, \text{ad}(E_G)\otimes\Omega^2_{X_0})\, ,
	$$
	which is an $\text{ad}(E_G)$-valued $2$-form on the groupoid $\XX$. 
	
	\begin{definition}
		For a  principal $G$-bundle ${\mathbb E}_G=[s^*E_G\rightrightarrows E_G]$ over a Lie groupoid $\XX=[X_1\rightrightarrows X_0]$ with a connection $D$, the section $K(D)$ is called the {\em curvature} of the connection $D$.
	\end{definition}
	
	The above constructions readily imply now the following result, which allows for the development of a general Chern-Weil theory and the construction of characteristic classes. This will be discussed systematically in \cite{BCKN2}.
	
	\vspace*{0.25cm}
	\begin{theorem}
		For any invariant form $\nu\,\in\, (\text{Sym}^k({\mathfrak g}^*))^G$, the evaluation
		$\nu(K(D))$ on the curvature $K(D)$ is a closed $2k$-form on the Lie groupoid $\XX$.
	\end{theorem}

	\subsection{Connections on principal $G$-bundles over differentiable stacks}
	Finally, we shall turn to the stacky picture and study connections for principal $G$-bundles on a differentiable stack (compare also 
	\cite{BMW}).
	
	\begin{definition}[Connection on principal bundle over differentiable stack]
		Let ${\mathscr E}_G$ be a principal $G$-bundle on a differentiable stack $\X$. A {\it connection} $\nabla$ on ${\mathscr E}_G$ 
		consists of the data of a connection $\nabla_u$ on each principal $G$-bundle ${\mathscr E}_{G, u}$, where $u\,:\,
		U\,\longrightarrow\, \X$ is a smooth atlas for $\X$, which pulls back naturally with respect to each $2$-commutative
		diagram of the form
		$$\xy \xymatrix{U\ar[dr]_{u}\ar[rr]^{\varphi}& & V\ar[dl]^v\\
			&{\mathscr X. }&}
		\endxy$$
		A connection $\nabla$ on ${\mathscr E}_G$ is {\it flat} or {\it integrable} if it is in addition integrable on each principal $G$-bundle ${\mathscr E}_{G, u}$.	\end{definition}
	
	Let us unravel this definition with more details. We can realize the connections for each atlas $u$ in terms of $\mathfrak 
	g$-valued $1$-forms such that  $\omega_{{\mathscr E}_{G, u}}\,\in\, H^0({{\mathscr E}_{G, u}},\, \Omega^1_{{\mathscr E}_{G, u}}\otimes \g)$ is the corresponding 
	$1$-form for the connection $\nabla_u$ on ${\mathscr E}_{G, u}$. From the Cartesian diagram
	\[
	\xy \xymatrix{\varphi^*\E_{G,v} \rrto^{\bar{\varphi}}\dto &&\E_{G, v}\dto \\
		U\rrto_{\varphi}& & V}
	\endxy
	\]
	we get a connection $\bar{\varphi}^*\omega_{\E_{G,v}}$ on $\varphi^*\E_{G,v}$. The condition for the existence of a connection is then given as follows
	$$\omega_{\E_{G, u}}\,=\, \theta_{\varphi, \alpha}^* \bar{\varphi}^* \omega_{\E_{G,v}}\, .$$
	
	From now on, we will assume that $\X$ is a Deligne--Mumford stack, which in particular means that the tangent stack 
	$T\X$ gives rise to a vector bundle $T\X\,\longrightarrow\, \X$ over $\X$. For a general differentiable stack, this is not always the case (see \cite{B1}, \cite{LTX}, \cite{Hp}).
	
	Given a principal $G$-bundle $\E_G$ over $\mathscr X$, we define the associated {\em Atiyah bundle} $\text{At}(\E_G)$ over $\X$ by setting
	$$\text{At}(\E_G)_u\,:=\,\text{At}(\E_{G, u})\, .$$
	for any \'etale morphism $u\,:\, U\,\longrightarrow\, \X$. In addition, for each $2$-commutative diagram of the form 
	\begin{equation}
	\xy \xymatrix{U\ar[dr]_{u}\ar[rr]^{\varphi}& & V\ar[dl]^v\\
		&{\mathscr X }&}
	\endxy
	\end{equation}
	where $u, v. \varphi$ are \'etale morphisms, we can compose the isomorphism $\text{At}(\E_{G, u})
	\,\stackrel{\cong}{\longrightarrow}\, \text{At}(\varphi^*\E_{G, v})$ with the isomorphism
	$\text{At}(\varphi^*\E_{G, v})\,\stackrel{\cong}{\longrightarrow}\, \varphi^*\text{At}(\E_{G, u})$ to get an isomorphism
	$$\text{At}(\E_G)_u \stackrel{\cong}\longrightarrow \varphi^*\text{At}(\E_G)_v.$$
	Similarly, given a principal $G$-bundle $\E_G$ over $\X$, we can define the associated {\it adjoint bundle}
	$\text{ad}(\E_G)$ over $\X$ arising from the adjoint representation of $G$ by setting (compare \cite[Sect. 1.4]{BMW})
	$$\text{ad}(\E_G)\,:=\, \E_G \times^G\mathfrak{g}\, ,$$
	where $\text{ad}(\E_G)_u\,=\,\text{ad}(\E_{G, u}).$
	We then obtain a commutative diagram of vector bundles
	
	$$
	\xymatrix{0\rto &\text{ad}(\E_G)_u\rto\dto & \text{At}(\E_G)_ u\rto\dto & TU \rto\dto& 0\\
		0\rto & \varphi^* \text{ad}(\E_G)_v\rto & \varphi^* \text{At}(\E_G)_v\rto & \varphi^* TV\rto &0.}
	$$
	which gives a well-defined short exact sequence of vector bundles over the Deligne--Mumford stack $\X$, called the 
	{\it Atiyah exact sequence associated to the principal $G$-bundle} $\E_G$
	$$
	\xymatrix{0\rto &\text{ad}(\E_G)\rto & \text{At}(\E_G)\rto & T\X \rto& 0.}
	$$
	
	\begin{remark} In the situation of a general differentiable stack $\X$, similarly as for the tangent stack, $\text{At}(\E_G)$ 
		and $\text{ad}(\E_G)$ will generally {\em not} give vector bundles over $\X$.
	\end{remark}
	
	From the constructions and considerations above, it follows now
	
	\begin{proposition}
		A principal $G$-bundle $\E_G$ over a Deligne--Mumford stack $\X$ admits a connection if and only if its associated Atiyah exact sequence has a splitting.
	\end{proposition}
	
	Finally, we have the following comparison theorem for the existence of a connection of a principal $G$-bundle over a 
	Deligne--Mumford stack:
	
	\begin{theorem}
		Giving a connection (respectively, flat connection) on a principal $G$-bundle $\E_G$ over a Deligne--Mumford stack $\X$ 
		with \'etale atlas $x\,:\, X_0\,\longrightarrow\, \X$ is equivalent to giving a connection (respectively, flat
		connection) on the associated principal $G$-bundle $\mathbb{E}_G$ over the groupoid $\mathbb{X}\,=\,
		[X_1 \rightrightarrows X_0]$. Giving a 
		connection (respectively, flat connection) on a principal $G$-bundle $\mathbb{E}_G\,=\,[s^*E_G\rightrightarrows E_G]$ 
		over an \'etale Lie groupoid $\mathbb{X}\,=\,[X_1 \rightrightarrows X_0]$ is equivalent to giving a connection 
		(respectively, flat connection) on the associated principal $G$-bundle $\E_G$ of $\mathbb{E}_G$-torsors over the 
		classifying stack $\mathscr{B}\mathbb{X}$ of $\mathbb{X}$-torsors.
	\end{theorem}
	
	\begin{proof} This follows from the categorical equivalence of Theorem \ref{thm1} between the category ${\mathscr B}un_G({\mathscr X})$ of principal $G$-bundles over $\X$ and the category ${\mathbb B}un_G({\mathbb X})$ of principal $G$-bundles over the associated Lie groupoid  $\mathbb{X}\,=\,[X_1 \rightrightarrows X_0]$ and the fact that a splitting of the Atiyah exact sequence 
		of associated vector bundles over the stack $\X$ 
		$$
		\xymatrix{0\rto &\text{ad}(\E_G)\rto & \text{At}(\E_G)\rto & T\X \rto& 0}
		$$
		corresponds to a splitting of the Atiyah exact sequence of associated vector bundles over the smooth space $X_0$
		\begin{equation*}
		0\,\longrightarrow\, \text{ad}(E_G)\,\longrightarrow\,\text{At}(E_G)\,
		{\longrightarrow}\, TX_0 \,\longrightarrow\, 0\,
		\end{equation*}
		and vice versa involving the $2$-Cartesian diagram 
		\[
		\xy \xymatrix{E_G\ar[r] \ar[d]& X_0\ar[d]\\
			\E_G\ar[r]^\pi & \X}
		\endxy
		\]
		and unraveling the explicit constructions of associated bundles as given in the proof of Theorem \ref{thm1}.
	\end{proof}
	
	Let us now consider the classifying stack $\BB^{\nabla}G$ of principal $G$-bundles with connections. The objects are 
	triples $(P,\,\omega,\, S)$ where $S$ is a smooth space of $\mathfrak S$ and $P$ is a principal $G$-bundle over $S$ and 
	$\omega\,\in\,\Omega^1(S,\, \mathfrak{g})^G$ a connection $1$-form. A morphisms $(P, \,\omega,\, S)\,\longrightarrow\, (P',\, 
	\omega', \,S')$ is given by a commutative diagram
	\[
	\xy \xymatrix{ P\ar[r]^{\varphi}\ar[d]&
		P'\ar[d]\\
		S\ar[r]& S'}
	\endxy
	\]
	where $\varphi\,:\, P\,\longrightarrow\, P'$ is a $G$-equivariant map and $\varphi^*\omega'=\omega$.
	Then ${\mathscr B}^{\nabla}G$ together with the projection functor $$\pi:
	{\mathscr B}^{\nabla} G\,\longrightarrow \,{\mathfrak S}\, ,\ ~ (P,\, \omega,\, S)\,\longmapsto\, S$$
	is a groupoid fibration over $\mathfrak S$. 
	We note that ${\mathscr B}^{\nabla}G$ is, in fact, a stack as principal $G$-bundles glue and connection $1$-forms on the principal bundles glue as well (compare also \cite{CLW}).
	
	\begin{remark} 
		It is not clear in general if and under which conditions ${\mathscr B}^{\nabla}G$ is actually a differentiable or Deligne--Mumford stack and not just a stack over $\mathfrak S$.
	\end{remark}
	
	The above constructions now allow us to characterize principal $G$-bundles over differentiable stacks equivalently as follows.
	
	\begin{proposition}\label{BNabla}
		Let $\X$ be a differentiable stack. Giving a principal $G$-bundle with connection over $\X$ is equivalent to giving a morphism
		of stacks $\X\,\longrightarrow\, {\mathscr B}^{\nabla}G$ and two principal $G$-bundles with connections over $\X$
		are isomorphic if and only if the corresponding morphisms of stacks $\X\,\longrightarrow\, {\mathscr B}^{\nabla}G$ are $2$-isomorphic.
	\end{proposition}
	
	Similarly, as before, we can consider the groupoid fibration ${\mathscr B}un_G^{\nabla}({\mathscr X})$ over $\mathfrak S$ whose objects over a smooth space $U$ are principal $G$-bundles ${\mathscr E}_G$ over ${\mathscr X} \times U$ with connections and whose morphisms are given by pullback diagrams of principal $G$-bundles with connections as above. From Proposition \ref{BNabla} we get therefore the following:
	
	\begin{proposition}
		There is an equivalence of groupoid fibrations over $\mathfrak S$
		$${\mathscr B}un_G^{\nabla}({\mathscr X})\,\cong\, {\mathscr H}om(\X,\, {\mathscr B}^{\nabla}G)\, .$$
	\end{proposition}
	
	Collier--Lerman--Wolbert \cite[Theorem 6.4]{CLW} proved a similar result involving holonomy and parallel transport for principal $G$-bundles over stacks.
	
	\section*{Acknowledgements}
	
The authors would like
to thank the anonymous referee for very useful comments and suggestions.
The first author acknowledges the support of a J. C. Bose Fellowship. The second author acknowledges research support 
from SERB, DST, Government of India grant MTR/2018/000528. The fourth author would like to thank the Tata Institute 
of Fundamental Research (TIFR) in Mumbai and the University of Leicester for financial support.


\begin{thebibliography}{ZZZZZ}
		
		\bibitem[SGA4]{SGA4} M. Artin, A. Grothendieck and J.-L. Verdier,
		Th\'eorie des topos et cohomolgie \'etale des sch\'emas, {\it
			S\'eminaire de g\'eometrie alg\'ebrique du Bois-Marie (SGA 4)},
		Lecture Notes in Mathematics, Vol. 269, 270, 305,
		Springer-Verlag, Berlin-New York, 1972-1973.
		
		\bibitem[Ao]{Ao} M. Aoki, Hom stacks, {\it Manuscripta Math.}, 119 (2006), 37--56.
		
		\bibitem[AC]{AC} C. Arias Abad and M. Crainic, Representations up to homotopy and Bott's 
		spectral sequence for Lie groupoids, {\it Adv. Math.}  248 (2013), 416--452.
		
		\bibitem[At]{At} M. F. Atiyah, Complex analytic connections in fiber
		bundles, {\it Trans. Amer. Math. Soc.} 85 (1957), 181--207.
		
		\bibitem[Be]{B1} K. Behrend, On the de Rham cohomology of differential and algebraic
		stacks, {\it Adv. Math.} 198 (2005), 583--622.
		
		\bibitem[BX]{BX} K. Behrend and P. Xu, Differentiable Stacks and Gerbes, {\it J. of Sympl.
			Geom.} 9 (2011), 285--341.
		
		\bibitem[BMW]{BMW} I. Biswas, S. Majumder and M. L. Wong, Root stacks, principal bundles, and
		connections. {\it Bull. Sci. Math.} 136 (2012), 369--398.
		
		\bibitem[BN]{BN} I. Biswas and F. Neumann, Atiyah sequences, connections and
		characteristic forms for principal bundles over groupoids and stacks,
		{\it Comptes Rendus Math. Acad. Sci. Paris} 352 (2014), 59--64.
		
		\bibitem[BCKN]{BCKN2}I. Biswas, S. Chatterjee, P. Koushik, and F. Neumann, Chern-Weil theory for principal bundles 
		over Lie groupoids and differentiable stacks, preprint, arXiv:2012.08447.
		
		\bibitem[BSS]{BSS} R. Bott, H. Shulman, and J. Stasheff, On the de Rham theory of certain
		classifying spaces, {\it Adv. Math.} 20 (1976), 43--56.
		
		\bibitem[CLW]{CLW} B. Collier, E. Lerman and S. Wolbert, Parallel transport on principal bundles over stacks,
		{\it Jour. Geom. Phy.} 107 (2016), 187--213.
		
		\bibitem[Co]{C} A. Connes, {\it Noncommutative Geometry}, Academic Press, Inc. 1994.
		
		\bibitem[CM]{CM} M. Crainic and I. Moerdijk, Foliation groupoids and their cyclic homology, {\it Adv. Math.} 157 (2001), 177--197.
		
		\bibitem[dH]{MDH} M. L. del Hoyo, Lie groupoids and differentiable stacks, {\it Portugaliae Mathematica} 70 (2013), 161--209.
		
		\bibitem[DE]{DE} T. Drummond and L. Egea, Differential forms with values in VB-groupoids and its Morita invariance,
		{\it Jour. of Geom. Phys.} {\bf 135} (2019), 42--69.
		
		\bibitem[Du1]{D1} J. Dupont, Simplicial de Rham cohomology and characteristic classes
		of flat bundles, {\it Topology} 15 (1976), 233--245.
		
		\bibitem[Du2]{D2} J. Dupont, {\it Curvature and characteristic
			classes}, Lecture Notes in Mathematics, Vol. 640
		Springer-Verlag, Berlin-New York, 1978.
		
		\bibitem[Fa]{Fa} B. Fantechi, Stacks for everybody, {\it European
			Congress of Mathematics, Barcelona 2000}, Vol. I, Birkh\"auser
		Verlag, 2001, 349--359.
		
		\bibitem[FN]{FN} M. Felisatti and F. Neumann, Secondary theories for
		\'etale groupoids, {\it Contemp. Math.} 571 (2012), 135--151.
		
		\bibitem[Gr]{Gr} A. Grothendieck, Revetements \'etale et groupe fundamental, {\it S\'eminaire de
			g\'eometrie alg\'ebrique du Bois-Marie 1960-1961 (SGA 1)}, Lecture Notes in
		Mathematics, Vol. 224, Springer-Verlag, Berlin-New York, 1971.
		
		\bibitem[He]{H} J. Heinloth, Notes on differentiable stacks, Mathematisches Institut
		Seminars 2004-2005, (Y. Tschinkel ed.), Georg-August Universit\"at G\"ottingen, 2005, 1--32.
		
		\bibitem[Hp]{Hp} R. Hepworth, Vector fields and flows on differentiable stacks, {\it Theory and Applications of 
			Categories} 22 (2009), 542--587.
		
		\bibitem[HO]{HO} S. Herrera and C. Ortiz, Private communication, 2020. 
				
		\bibitem[LGSX]{LSX} C. Laurent-Gengoux, P. Sti\'enon and P. Xu, Non-abelian differentiable
		gerbes, {\it Adv. Math.} 220 (2009), 1357--1427.
		
		\bibitem[LGTX]{LTX} C. Laurent-Gengoux, J.-L. Tu and P. Xu, Chern-Weil map for principal
		bundles over groupoids, {\it Math. Zeit.} 255 (2007), 451--491. 
		
		\bibitem[L]{L} E. Lerman, Orbifolds as stacks?, {\it L' Enseignement Math.} {\bf 56} (2010), 315--363.		
		
		\bibitem[LM]{LM} E. Lerman and A. Malkin, Hamiltonian group actions on symplectic
		Deligne-Mumford stacks and toric orbifolds, {\it Adv. Math.} 229 (2012), 984--1000.

		\bibitem[Ne]{N} F. Neumann, Algebraic Stacks and Moduli of Vector Bundles. {\it Publica\c{c}oes Matematicas}, 
		Research Monograph, IMPA Rio de Janeiro 2009 (2nd ed. 2011), 142pp. 
		
		\bibitem[Ol]{Ol} M. C. Olsson, Hom-stacks and restriction of scalars, {\it Duke Math. Journal} 134 (2006), 139--164.
		
		\bibitem[Pi]{P} M. A. Salazar Pinz\'on, Pfaffian groupoids, PhD thesis, Utrecht 2013, arXiv:1306.1164v2 [math.DG].
		
		\bibitem[Ta]{T} X. Tang, Deformation quantization of pseudo-symplectic (Poisson)
		groupoids, {\it Geom. Funct. Anal.} 16 (2006), 731--766.
		
		\bibitem[Tr]{Tr} G. Trentinaglia, Regular Cartan groupoids and longitudinal representations, {\it Adv. Math.} {\bf 340} (2018), 1--47.
		
		\bibitem[TXLG]{TXL} L. Tu, P. Xu and C. Laurent-Gengoux, Twisted K-theory of
		differentiable stacks, {\it Ann. Scient. \'Ec. Norm. Sup.} 37 (2004), 841--910.
		
		\bibitem[Vi]{Vi} A. Vistoli, Grothendieck topologies, fibered categories, and descent theory. In:{\it Fundamental algebraic geometry}, 
		Math. Surveys. Monogr., vol. 123, 1--104, Amer. Math. Soc., Providence, RI, 2005.
		
	\end{thebibliography}
\end{document}